\documentclass{amsart}
\usepackage[utf8]{inputenc}
\usepackage{geometry}
\usepackage{amsmath}
\usepackage{amssymb}
\usepackage{amsthm}
\usepackage{bm}
\usepackage{comment}
\usepackage[dvipsnames]{xcolor}
\usepackage{tikz}
\usepackage{lmodern}
\usetikzlibrary{arrows.meta,graphs,graphs.standard,quotes,shapes.misc}

\usepackage{caption}
\usepackage{blkarray}
\usepackage{multirow}
\usepackage{hhline}
\usepackage{nicematrix}
\usepackage{arydshln}

\usepackage[pdftex]{hyperref}
\hypersetup{
    colorlinks=true,
    linkcolor={MyBlue},
    citecolor={magenta},
    urlcolor={MyBlue}
}

\newtheorem{defn0}{Definition}[section]
\newtheorem{prop0}[defn0]{Proposition}
\newtheorem{thm0}[defn0]{Theorem}
\newtheorem{assump0}[defn0]{Assumption}
\newtheorem{lemma0}[defn0]{Lemma}
\newtheorem{corollary0}[defn0]{Corollary}
\newtheorem{example0}[defn0]{Example}
\newtheorem{remark0}[defn0]{Remark}
\newtheorem{conjecture0}[defn0]{Conjecture}

\newtheorem*{assumption*}{Assumption}
\newtheorem{question*}{Question}

\newenvironment{definition}{\medskip \begin{defn0}}{\end{defn0}}
\newenvironment{proposition}{\medskip \begin{prop0}}{\end{prop0}}
\newenvironment{theorem}{\medskip \begin{thm0}}{\end{thm0}}

\newenvironment{lemma}{\medskip \begin{lemma0}}{\end{lemma0}}
\newenvironment{corollary}{\medskip \begin{corollary0}}{\end{corollary0}}
\newenvironment{example}{\medskip \begin{example0}\rm}{\end{example0}}
\newenvironment{remark}{ \medskip\begin{remark0}\rm}{\end{remark0}}
\newenvironment{conjecture}{\medskip\begin{conjecture0}}{\end{conjecture0}}

\definecolor{MyBlue}{RGB}{0,101,189} 
\definecolor{MyRed}{RGB}{234, 114, 55} 
\definecolor{MyGreen}{RGB}{162,173,0}
\definecolor{MyYellow}{RGB}{246, 235, 97} 

\DeclareMathOperator{\pa}{pa}
\DeclareMathOperator{\ch}{ch}
\DeclareMathOperator{\rk}{rank}
\DeclareMathOperator{\jpa}{jpa}
\DeclareMathOperator{\im}{Im}
\DeclareMathOperator{\initial}{in}
\DeclareMathOperator{\PD}{PD}

\usepackage[nameinlink,capitalise]{cleveref}

\title{Algebraic Sparse Factor Analysis}
\author{Mathias Drton}
\address{Munich Center for Machine Learning (MCML) and Department of Mathematics, School of Computation, Information and Technology, Technical University of Munich} 
\email{mathias.drton@tum.de}

\author{Alexandros Grosdos}
\address{Institute of Mathematics, University of Augsburg} 
\email{alexandros.grosdos@uni-a.de}

\author{Irem Portakal}
\address{Max Planck Institute for Mathematics in the Sciences, Leipzig} 
\email{mail@irem-portakal.de}

\author{Nils Sturma}
\address{Munich Center for Machine Learning (MCML) and Department of Mathematics, School of Computation, Information and Technology, Technical University of Munich} 
\email{nils.sturma@tum.de}

\keywords{factor analysis model, dimension, join of ideals, Gröbner basis, toric, edge ideal, hypergraph}
\subjclass{62H25, 62R01, 13F65, 14M25, 14N07} 

\begin{document}
\begin{abstract}
    Factor analysis is a statistical technique that explains correlations among observed random variables with the help of a smaller number of unobserved factors. In traditional full factor analysis, each observed variable is influenced by every factor. However, many applications exhibit interesting sparsity patterns, that is, each observed variable only depends on a subset of the factors. In this paper, we study such sparse factor analysis models from an algebro-geometric perspective. Under mild conditions on the sparsity pattern, we examine the dimension of the set of covariance matrices that corresponds to a given model. Moreover, we study algebraic relations among the covariances in sparse two-factor models. In particular, we identify cases in which a Gröbner basis for these relations can be derived via a 2-delightful term order and join of toric ideals of graphs.
\end{abstract}
\maketitle

\section{Introduction}\label{sec: introduction}

Factor analysis provides powerful statistical tools to analyze complex data
by representing a possibly large number of dependent random variables
as linear functions of a smaller number of underlying source variables, the factors.
Techniques from factor analysis have found widespread application in a variety of fields, including psychology \cite{horn1965arationale, reise2000factor, caprara1993thebig}, econometrics \cite{fan2008high, amann2016bayesian}, education \cite{schreiber2006reporting, beavers2013practical}, and epidemiology \cite{martines1998invited, santos2019principal}.  

Factor analysis models may be defined as follows. Consider an observed random vector $X=~(X_v)_{v \in V}$ that is indexed by a finite set $V$ and a vector  $Y=(Y_h)_{h \in \mathcal{H}}$ of unobserved random variables, called \emph{factors}, that is indexed by a finite set $\mathcal{H}$. In applications, the number of factors $|\mathcal{H}|$ is usually smaller than the number of observed variables $|V|$.
The factor analysis model postulates that the observed variables are linear functions of the factors and noise, i.e.,
\[
    X = \Lambda Y + \varepsilon,
\]
where $\Lambda=(\lambda_{vh}) \in \mathbb{R}^{|V| \times |\mathcal{H}|}$ is an unknown coefficient matrix, known as \emph{factor loading matrix}.  The noise $\varepsilon=(\varepsilon_v)_{v \in V}$ is comprised of independent random variables with mean zero and positive variance; so $\mathbb{E}[\varepsilon_v]=0$ and $\text{Var}[\varepsilon_v]=:\omega_{vv}\in(0,\infty)$.
The latent (unobserved) factors $(Y_h)_{h \in \mathcal{H}}$ are assumed to be mutually independent, and also independent of the noise $\varepsilon$. Without loss of generality, we fix the scale of the latent factors such that each $Y_h$ has mean zero and variance one. We emphasize that while we assume that the second-order moments (and thus covariances) exist, no further assumptions are made about the type of the noise random variables allowed in the model.
The main object of study is now the covariance matrix $\Sigma$ of the observed random vector $X$, which is given by
\begin{equation} \label{eq:covariance-matrix}
    \Sigma := \text{Cov}[X] = \Lambda \Lambda^{\top} + \Omega,
\end{equation}
where $\Omega$ is a diagonal matrix with entries $\omega_{vv} = \text{Cov}[\varepsilon_v]$. In traditional full factor analysis, all coefficients $\lambda_{vh}$ are nonzero \cite{anderson1956statistical}. Full factor analysis models were studied from a computational algebraic geometry point of view in \cite{AlgebraicFactorAnalysis}, where Gröbner bases were used to investigate the \emph{{ideal of invariants}} that vanish on the space of covariance matrices. The generators emerge from rank conditions on the symmetric covariance matrix under elimination of the diagonal entries.

The journey of this paper extends {\emph{beyond}} \cite{AlgebraicFactorAnalysis}, prompting a study of \emph{sparse} factor analysis models under an algebro-geometric perspective. Recently, there has been considerable interest in sparse factor analysis models, which posit that some (or often many) of the coefficients $\lambda_{vh}$ are equal to zero.  Examples of recent research on sparsity include work on correlation thresholding \cite{kim2022correlation}, $l_1$-penalization \cite{lan2014Sparse, trandafilov2017sparse}, and Bayesian approaches \cite{fruehwirthschnatter2018sparse, ohn2023bayesian}. Moreover, sparse factor analysis models are the building block for many directed graphical models with latent variables \cite{bollen1989structural, barber2022halftrek} that have applications in causality \cite{pearl2000causality, peters2017elements}. In this context, the coefficients (or factor loadings) $\lambda_{vh}$ can be interpreted as causal effects of the latent variables $Y_h$ on the observed variables $X_v$. To represent sparsity assumptions, it is useful to adopt a graphical perspective and encode a zero pattern in $\Lambda$ by a directed graph with nodes $V \cup \mathcal{H}$ \cite{maathuis2019handbook}. For an observed node $v \in V$ and a latent node $h \in \mathcal{H}$, the coefficient $\lambda_{vh}$ is allowed to be nonzero only if the edge $h \rightarrow v$ is in the graph. 

\begin{example}
\label{ex:LA}
    We revisit a study from \cite[p.\ 14]{harman1976modern} that pertains to five socio-economic variables that are observed in twelve districts in the greater Los Angeles area:  total population, median school years, total employment, miscellaneous professional services, and median house value. 
    Applying $l_1$-penalization techniques to the data, a model corresponding to the graph $G$ in Figure~\ref{fig:harman-example} is found by \cite[Table 1, Column 3]{trandafilov2017sparse}. The model imposes, for example, that total population is independent of total employment given only the first latent variable. In this model, the factor loading matrix has the zero pattern
        \[
        \Lambda = \scriptstyle\begin{pmatrix}
            \lambda_{11} & 0 & \lambda_{31} & \lambda_{41} & \lambda_{51} \\
            0 & \lambda_{22} & 0 & \lambda_{42} & \lambda_{52} \\
        \end{pmatrix}^{\top},
    \]
    and gives rise to the covariance matrix $\Sigma \in F(G)$ of the form
\[
\Sigma = (\sigma_{uv}) =\scriptstyle
\begin{pmatrix}    \omega_{11}+\lambda_{11}^{2}&0 &\lambda_{11}\lambda_{31}&\lambda_{11}\lambda_{41}&\lambda_{11}\lambda_{51}\\
0 &\omega_{22}+\lambda_{22}^{2}& 0 &\lambda_{22}\lambda_{42}&\lambda_{22}\lambda_{52}\\
\lambda_{11}\lambda_{31}&0&\omega_{33}+\lambda_{31}^{2}&\lambda_{31}\lambda_{41}&\lambda_{31}\lambda_{51}\\
\lambda_{11}\lambda_{41}&\lambda_{22}\lambda_{42}&\lambda_{31}\lambda_{41}&\omega_{44}+\lambda_{41}^{2}+\lambda_{52}^{2}&\lambda_{41}\lambda_{51}+\lambda_{42}\lambda_{52}\\
\lambda_{11}\lambda_{51}&\lambda_{22}\lambda_{52}&\lambda_{31}\lambda_{51}&\lambda_{41}\lambda_{51}+\lambda_{42}\lambda_{52}&\omega_{55}+\lambda_{51}^{2}+\lambda_{52}^{2}\\
\end{pmatrix}.
\]
The factor analysis model $F(G)$ is $12$-dimensional which is equal to the expected dimension obtained from counting parameters. However, we will show in this paper that the dimension of sparse factor models is not always equal to the number of parameters. The (toric) ideal of variants $I(G)$ is generated by two monomials and one binomial: $$\langle \sigma_{12}, \sigma_{23}, \sigma_{15} \sigma_{34} - \sigma_{14} \sigma_{35}  \rangle.$$    
\end{example}

\tikzset{
      every path/.style={thick}
}

\begin{figure}
\begin{tikzpicture}[align=center]
    \node[style={circle, inner sep=0.3mm, minimum size=0.7cm, draw, thick, black, fill=lightgray, text=black}] (h1) at (0,1.25) {$h_1$};
    \node[style={circle, inner sep=0.3mm, minimum size=0.7cm, draw, thick, black, fill=lightgray, text=black}] (h2) at (1,-1.25) {$h_2$};
    
    \node[style={rounded rectangle,draw}] (1) at (-4,0) {Pop};
    \node[style={rounded rectangle,draw}] (2) at (-2,0) {School};
    \node[style={rounded rectangle,draw}] (3) at (0,0) {Employ};
    \node[style={rounded rectangle,draw}] (4) at (2,0) {Service};
    \node[style={rounded rectangle,draw}] (5) at (4,0) {House};
    
    \draw[MyBlue] [-latex] (h1) edge (1);
    \draw[MyBlue] [-latex] (h1) edge (3);
    \draw[MyBlue] [-latex] (h1) edge (4);
    \draw[MyBlue] [-latex] (h1) edge (5);
    \draw[MyRed] [-latex] (h2) edge (2);
    \draw[MyRed] [-latex] (h2) edge (4);
    \draw[MyRed] [-latex] (h2) edge (5);
\end{tikzpicture}
\caption{Graph corresponding to a sparse factor analysis model for a study on metropolitan districts. Gray nodes correspond to latent nodes.}
\label{fig:harman-example}
\end{figure}
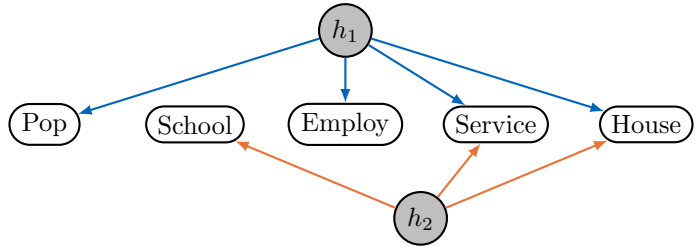

\tikzset{
      every node/.style={circle, inner sep=0.3mm, minimum size=0.5cm, draw, thick, black, fill=white, text=black},
      every path/.style={thick}
}

The binomials of this form are known as {\emph{tetrads}} in statistics which reflects the fact that the polynomial arises with four observed random variables. They also arise as generators of the ideal of invariants for one-factor analysis models or, equivalently, as the toric ideal of the edge subring of complete graphs \cite{SecondHypersimplex}; compare also \cite[Cor.~6.5]{sullivant2008}. Harman \cite[p.77]{harman1976modern} highlighted the absence of knowledge regarding the ideal of invariants for models involving two or more factors. This was subsequently addressed in the context of full factor analysis models in \cite{AlgebraicFactorAnalysis}. The ideal of invariants can enhance useful statistics for testing goodness-of-fit; see, e.g., \cite{bollen2000atetrad, silva2006learning, AlgebraicFactorAnalysis, drton2016wald, sturma2022testing}. In this paper, we address this gap of knowledge for sparse~factor analysis models. \\

The organization and the main results of the paper are as follows: In Section \ref{sec:dimension}, we study the dimension of factor analysis models. First, we give a general upper bound on the dimension in Theorem~\ref{thm:dimension-general}, which reveals that sparse factor models can be defective, that is, they may not have expected dimension. This differs from full factor analysis models that are always of expected dimension. For models that exhibit a minimal level of sparsity, which we call the zero upper triangular assumption (ZUTA), we also provide a lower bound on the dimension in Theorem~\ref{thm:dim-lower-bound}. The assumption ensures that the rows and columns of the matrix $\Lambda$ can be permuted such that the upper triangle of the matrix is zero. In many cases, the upper and lower bounds coincide, and we obtain a combinatorial formula for the dimension.

In Section \ref{sec:invariants}, we study the ideal of invariants of sparse one- and two-factor analysis models. We identify the ideal of sparse one-factor models as the toric ideal of the edge subring of complete graphs with isolated vertices and, consequently, provide a reduced Gröbner basis (Proposition~\ref{prop: Gröbner basis for sparse one-factor}). In Theorem~\ref{thm:variety-of-model}, we characterize the Zariski closure of two-factor analysis models. Moreover, we give an explicit description for the generators of Gr\"obner bases with respect to any circular term order for a subclass of sparse two-factor analysis models in Theorem~\ref{thm: a Gröbby}. This result uses the 2-delightful strategy that was introduced in \cite{combinatorialsecant} for secant varieties. We generalize this to joins of sparse one-factor analysis models, i.e., we study joins of toric ideals of the edge subring of complete graphs with isolated vertices. These ideals exhibit interesting combinatorial aspects, such as the initial ideal, which can be realized as the monomial edge ideal of a hypergraph (Lemma~\ref{lem: initial ideal of a hypergraph}).
Finally, our work opens up some natural conjectures regarding Gröbner bases of sparse-factor models that we outline in Section~\ref{subsec:computational-results}.
 
 The supplementary code for our results can be found on the MathRepo page: 
\begin{center}
    \url{https://mathrepo.mis.mpg.de/sparse-factor-analysis}
\end{center}

\section{Dimension} \label{sec:dimension}

Let $G=(V \cup \mathcal{H}, D)$ be a directed graph, where $V$ and $\mathcal{H}$ are finite disjoint sets of observed and latent nodes. We assume that the graph $G$ only contains edges that point from latent to observed nodes, that is, $D \subseteq \mathcal{H} \times V$; see Figure~\ref{fig:graph_5_2_4} for an example with $H = \{h_1, h_2\}$ and $V = \{v_1, \ldots, v_7\}$.  We refer to such bipartite graphs as \emph{factor analysis graphs}.  If $(h,v) \in D$, which we also denote by $h \rightarrow v \in D$, then $h\in \mathcal{H}$ is a parent of its child $v\in V$.  The respective sets of all parents and children are denoted by $\pa(v) = \{h \in \mathcal{H}: h \rightarrow v \in D\}$ and  $\ch(h)=\{v \in V: h \rightarrow v \in D\}$.

Every factor analysis graph determines a factor analysis model that for our purposes may conveniently be identified with the set of its covariance matrices.  For a definition, we let $\mathbb{R}^D$ denote the set of real $|V| \times |\mathcal{H}|$ matrices $\Lambda = (\lambda_{vh})$ with support $D$, that is, $\lambda_{vh} = 0$ if $h \rightarrow v \not\in D$.  Furthermore, we write $\PD(p)$ for the cone of positive definite $p \times p$ matrices, and $\mathbb{R}^p_{>0}\subset \PD(p)$ for the subset of diagonal positive definite matrices.

\begin{definition} \label{def:model}
Let $G=(V \cup \mathcal{H}, D)$ be a factor analysis graph with $|V|=p$ and $|\mathcal{H}|=m$.  As a model of the covariance matrix, the factor analysis model determined by $G$ is the image $F(G)=\im(\tau_G)$ of the parametrization map
\begin{align}\label{eq: parametrization}
\begin{split}
    \tau_G : \mathbb{R}^p_{>0} \times \mathbb{R}^D &\longrightarrow \PD(p)  \\
    (\Omega, \Lambda) &\longmapsto \Omega  + \Lambda \Lambda^{\top}.
\end{split}
\end{align}
\end{definition}

The covariance model $F(G)$ is a parameterized subset of the $\binom{p+1}{2}$-dimensional space of symmetric $p \times p$ matrices, and its dimension is the maximal rank of the Jacobian matrix of the map $\tau_G$ in Definition~\ref{def:model}.  Naturally, the \emph{expected dimension} of  $F(G)$ is equal to $\min\{|V|+|D|, \binom{|V|+1}{2}\}$, the minimum of the number of parameters in $(\Omega,\Lambda)$ and the dimension of the ambient space. 

\begin{example}\label{ex: parametrizedsigma}
The graph in Figure~\ref{fig:graph_5_2_4} corresponds to a sparse model with $|V|=7$ nodes and $|D|=9$ edges. 
To simplify notation, we identify $v_1, \ldots, v_7$ with the integers $1, \ldots, 7$, and $h_1,h_2$ with the integers $1,2$. Then, the sparse factor loading matrix is 
\begin{equation} \label{eq:lambda-example}
\Lambda = \scriptstyle
\begin{pmatrix}
    \lambda_{11} & \lambda_{21} & \lambda_{31} & \lambda_{41} & \lambda_{51} & 0 & 0 \\
    0 & 0 & 0 & \lambda_{42} & \lambda_{52} & \lambda_{62} & \lambda_{72} \\
\end{pmatrix}^{\top},
\end{equation}
which gives rise to the covariance matrix $\Sigma \in F(G)$ of the form
\[
\Sigma = \scriptstyle
\begin{pmatrix}    \omega_{11}+\lambda_{11}^{2}&\lambda_{11}\lambda_{21}&\lambda_{11}\lambda_{31}&\lambda_{11}\lambda_{41}&\lambda_{11}\lambda_{51}&0&0\\
\lambda_{11}\lambda_{21}&\omega_{22}+\lambda_{21}^{2}&\lambda_{21}\lambda_{31}&\lambda_{21}\lambda_{41}&\lambda_{21}\lambda_{51}&0&0\\
\lambda_{11}\lambda_{31}&\lambda_{21}\lambda_{31}&\omega_{33}+\lambda_{31}^{2}&\lambda_{31}\lambda_{41}&\lambda_{31}\lambda_{51}&0&0\\
\lambda_{11}\lambda_{41}&\lambda_{21}\lambda_{41}&\lambda_{31}\lambda_{41}&\omega_{44}+\lambda_{41}^{2}+\lambda_{42}^{2}&\lambda_{41}\lambda_{51}+\lambda_{42}\lambda_{52}&\lambda_{42}\lambda_{62}&\lambda_{42}\lambda_{72}\\
\lambda_{11}\lambda_{51}&\lambda_{21}\lambda_{51}&\lambda_{31}\lambda_{51}&\lambda_{41}\lambda_{51}+\lambda_{42}\lambda_{52}&\omega_{55}+\lambda_{51}^{2}+\lambda_{52}^{2}&\lambda_{52}\lambda_{62}&\lambda_{52}\lambda_{72}\\
0&0&0&\lambda_{42}\lambda_{62}&\lambda_{52}\lambda_{62}&\omega_{66}+\lambda_{62}^{2}&\lambda_{62}\lambda_{72}\\
0&0&0&\lambda_{42}\lambda_{72}&\lambda_{52}\lambda_{72}&\lambda_{62}\lambda_{72}&\omega_{77}+\lambda_{72}^{2}\\
\end{pmatrix}.
\]
The expected dimension of the corresponding model is equal to $|V|+|D|=16$, and as we verify in Corollary~\ref{cor:dimension-formula}, this is indeed the dimension of the model.
\end{example}

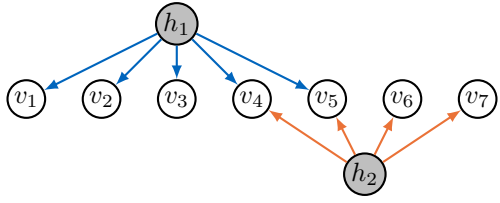
\begin{figure}[t]
\centering
    \begin{tikzpicture}
        \node[fill=lightgray] (h1) at (0,1) {$h_1$};
        \node[fill=lightgray] (h2) at (2.5,-1) {$h_2$};
        
        \node[] (1) at (-2,0) {$v_1$};
        \node[] (2) at (-1,0) {$v_2$};
        \node[] (3) at (0,0) {$v_3$};
        \node[] (4) at (1,0) {$v_4$};
        \node[] (5) at (2,0) {$v_5$};
        \node[] (6) at (3,0) {$v_6$};
        \node[] (7) at (4,0) {$v_7$};
        
        \draw[MyBlue] [-latex] (h1) edge (1);
        \draw[MyBlue] [-latex] (h1) edge (2);
        \draw[MyBlue] [-latex] (h1) edge (3);
        \draw[MyBlue] [-latex] (h1) edge (4);
        \draw[MyBlue] [-latex] (h1) edge (5);
        \draw[MyRed] [-latex] (h2) edge (4);
        \draw[MyRed] [-latex] (h2) edge (5);
        \draw[MyRed] [-latex] (h2) edge (6);
        \draw[MyRed] [-latex] (h2) edge (7);
    \end{tikzpicture}
    \caption{Factor analysis graph with $2$ latent factors and $7$ observed variables.}
    \label{fig:graph_5_2_4}
\end{figure}

If $G=(V \cup \mathcal{H}, D)$ is a factor analysis graph with all possible edges, so $D=\mathcal{H} \times V$, then the corresponding covariance model recovers the \emph{full} factor analysis model \cite{AlgebraicFactorAnalysis, anderson1956statistical}. However, using orthogonal transformations as in the QR decomposition, any covariance matrix $\Sigma$ in a full factor analysis model can be written as $\Sigma = \Omega + \Lambda \Lambda^{\top}$ such that the upper triangle of $\Lambda=(\lambda_{vh})$ is zero, i.e., $(\lambda_{vh})_{v < h} = 0$  if the nodes are given by $V=\{1, \ldots, p\}$ and $\mathcal{H} =\{1, \ldots, m\}$. Hence, any full factor model is equivalent to a sparse factor analysis model where only the edges corresponding to the upper triangle in $\Lambda$ are removed from the complete bipartite graph. Said differently, we obtain a graph that belongs to the set of factor analysis graphs satisfying the following assumption.

\begin{assumption*}[ZUTA] \label{ass:1-child}
A factor analysis graph and its associated model satisfy the \emph{Zero Upper Triangular Assumption (ZUTA)} if there exists a relabeling of the latent nodes $\mathcal{H}=\{h_1, \ldots, h_m\}$ such that $\ch(h_i)$ is not contained in $\bigcup_{j > i} \ch(h_j)$ for all $i=1,\ldots,m$.  In this case, there is then a relabeling of the observed nodes $V=\{v_1, \ldots, v_p\}$ such that $v_i \in \ch(h_i)$ and $v_i \not\in \bigcup_{j > i} \ch(h_j)$ for all $i=1,\ldots,m$.  
\end{assumption*}

ZUTA ensures that the rows and columns of the factor loading matrix $\Lambda$ can be permuted such that the upper triangle of the matrix is zero. 

\begin{example}
    The graph in Figure~\ref{fig:graph_5_2_4} satisfies ZUTA.  The latent nodes $h_1$ and $h_2$ are already ordered as desired.  A ZUTA labeling of $V$ is obtained if we permute, for example, the labelings of $v_2$ and $v_4$. This corresponds to permuting rows $2$ and $4$ of the parameter matrix $\Lambda$ in Equation~\eqref{eq:lambda-example}.
\end{example}

Note that ZUTA requires that $p \geq m$ and that each latent node has at least one observed child.  However, isolated latent nodes need not be considered as they only add a zero column in $\Lambda$.

\begin{remark}
In the special case where a factor analysis graph contains an observed node $v \in V$ such that $\pa(v) = \emptyset$, the dimension of the model is by one larger than the dimension of the model corresponding to the smaller graph where this node is removed.
\end{remark}

\begin{remark} \label{rem:pure-children}
ZUTA is more general than the ``$k$-pure-children'' condition that is often employed in previous work on structure identifiability of sparse factor analysis models \cite{arora2012learning, bing2020adaptive, moran2022identifiable, markham2023neurocausal}.
The $k$-pure-children condition requires that each latent node $h \in \mathcal{H}$ has at least $k$ (pure) children that have no other parents than $h$. In particular, the $1$-pure child condition implies that there is an upper $m \times m$ matrix inside $\Lambda$ which is diagonal. Hence, any $k$-pure children condition with $k \geq 1$ implies ZUTA.
Note that ZUTA also requires that each latent node $h \in \mathcal{H}$ has at least $1$ child, but the children are allowed to have more parents. For example, after relabeling such that ZUTA is satisfied, node $v_3$ needs to be a child of $h_3$, but it could also have $h_1$ and $h_2$ as parents. Only the first node $v_1$ has to be a \emph{pure} child of $h_1$. 
Conversely, a factor analysis graph in which there is no latent node that has a pure child does not satisfy ZUTA.
\end{remark}

It is proved in \cite{AlgebraicFactorAnalysis} that the dimension of full factor analysis models is always equal to the expected dimension obtained by counting parameters. Since full factor models are equivalent to models satisfying ZUTA, the number of edges is equal to $|D|=pm-\binom{m}{2}$, which implies that the expected dimension is given by $\min\{p(m+1)-\binom{m}{2}, \binom{p+1}{2}\}$. 

\begin{example}
    Consider the full factor analysis model with $m=2$ latent nodes and $p=5$ observed nodes. The model is equivalent to the model $F(G)$ corresponding to the graph in Figure \ref{fig:dense-graph}, and the dimension is equal to the number of parameters, that is, $\dim(F(G)) = 14$.
\end{example}

\begin{figure}[t]
\begin{center}
\begin{tikzpicture}[align=center]
    \node[fill=lightgray] (h1) at (-0.5,1) {$h_1$};
    \node[fill=lightgray] (h2) at (0,-1) {$h_2$};
    
    \node[] (1) at (-2.5,0) {$v_1$};
    \node[] (2) at (-1.5,0) {$v_2$};
    \node[] (3) at (-0.5,0) {$v_3$};
    \node[] (4) at (0.5,0) {$v_4$};
    \node[] (5) at (1.5,0) {$v_5$};
    
    \draw[MyBlue] [-latex] (h1) edge (1);
    \draw[MyBlue] [-latex] (h1) edge (2);
    \draw[MyBlue] [-latex] (h1) edge (3);
    \draw[MyBlue] [-latex] (h1) edge (4);
    \draw[MyBlue] [-latex] (h1) edge (5);
    \draw[MyRed] [-latex] (h2) edge (2);
    \draw[MyRed] [-latex] (h2) edge (3);
    \draw[MyRed] [-latex] (h2) edge (4);
    \draw[MyRed] [-latex] (h2) edge (5);
\end{tikzpicture}
\caption{A graph whose associated model is equal to the full two-factor analysis model with $5$ observed nodes.}
\label{fig:dense-graph}
\end{center}
\end{figure}
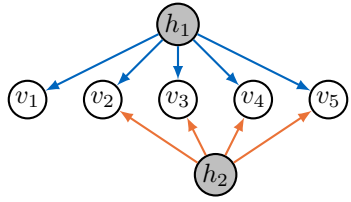

In Example \ref{ex: parametrizedsigma}, we saw a \emph{sparse} model that is also of expected dimensions. However, sparse factor analysis models differ fundamentally from full factor analysis models in the sense that their dimension is not always equal to the expected dimension. The next example shows models where the dimension drops, that is, the dimension is strictly smaller than the expected dimension. In particular, it is not enough to look at the graph and count parameters or zeros in the covariance matrix to tell the dimension.

\begin{example}
    Consider the three graphs in Figure~\ref{fig:dimension}. The expected dimension of the model $F(G)$ corresponding to graph (a) is $|V|+|D|=14$. On the other hand, the model is a subset of the space of symmetric matrices that has dimension $\binom{p+1}{2}=15$, and every covariance matrix $\Sigma=(\sigma_{vw}) \in F(G)$ has three zeros, $\sigma_{v_1 v_4}=\sigma_{v_1 v_5} = \sigma_{v_2 v_5} = 0$. Thus, we obtain $15-3=12$ as a trivial upper bound for the dimension. It turns out that we have indeed $\dim(F(G)) = 12$. However, the model corresponding to graph (b), obtained by adding one more node, shows that counting zeros in the covariance matrix is not enough. In this case, we have $\binom{p+1}{2}=21$ and there are five zeros in every covariance matrix in the model, namely $\sigma_{v_1 v_4}=\sigma_{v_1 v_5}=\sigma_{v_1 v_6}=\sigma_{v_2 v_5}=\sigma_{v_2 v_6}=0$. Thus we obtain an upper bound of $16$ for the dimension that is also equal to the expected dimension $|V|+|D|$. Nevertheless, the true dimension is given by $\dim(F(G))=15$. The model corresponding to graph (c) has a similar drop of dimension. In this case there are no zeros in the covariance matrix and the expected dimension is $|V|+|D|=18$, but the true dimension is $\dim(F(G))=17$.
\end{example}

\begin{center}
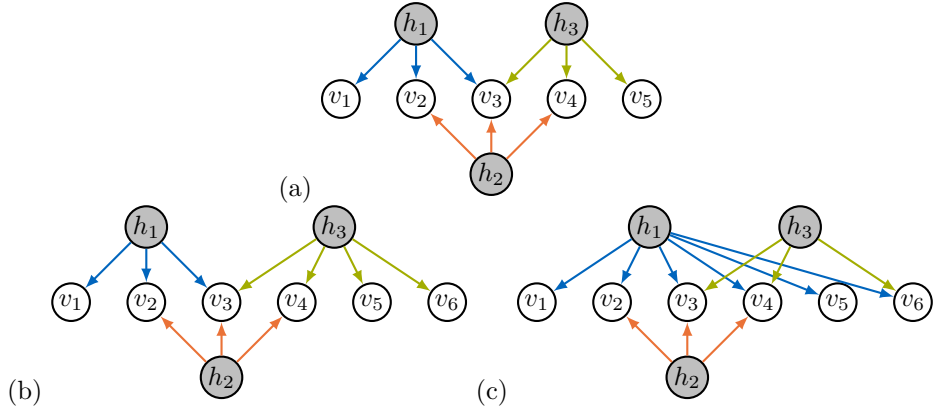
\begin{figure}[t]
(a)
\begin{tikzpicture}[align=center]
    \node[fill=lightgray] (h1) at (-1,1) {$h_1$};
    \node[fill=lightgray] (h2) at (0,-1) {$h_2$};
    \node[fill=lightgray] (h3) at (1,1) {$h_3$};
    
    \node[] (1) at (-2,0) {$v_1$};
    \node[] (2) at (-1,0) {$v_2$};
    \node[] (3) at (0,0) {$v_3$};
    \node[] (4) at (1,0) {$v_4$};
    \node[] (5) at (2,0) {$v_5$};
    
    \draw[MyBlue] [-latex] (h1) edge (1);
    \draw[MyBlue] [-latex] (h1) edge (2);
    \draw[MyBlue] [-latex] (h1) edge (3);
    \draw[MyRed] [-latex] (h2) edge (2);
    \draw[MyRed] [-latex] (h2) edge (3);
    \draw[MyRed] [-latex] (h2) edge (4);
    \draw[MyGreen] [-latex] (h3) edge (3);
    \draw[MyGreen] [-latex] (h3) edge (4);
    \draw[MyGreen] [-latex] (h3) edge (5);
\end{tikzpicture}\\
(b)
\begin{tikzpicture}[align=center]
    \node[fill=lightgray] (h1) at (-1.5,1) {$h_1$};
    \node[fill=lightgray] (h2) at (-0.5,-1) {$h_2$};
    \node[fill=lightgray] (h3) at (1,1) {$h_3$};
    
    \node[] (1) at (-2.5,0) {$v_1$};
    \node[] (2) at (-1.5,0) {$v_2$};
    \node[] (3) at (-0.5,0) {$v_3$};
    \node[] (4) at (0.5,0) {$v_4$};
    \node[] (5) at (1.5,0) {$v_5$};
    \node[] (6) at (2.5,0) {$v_6$};
    
    \draw[MyBlue] [-latex] (h1) edge (1);
    \draw[MyBlue] [-latex] (h1) edge (2);
    \draw[MyBlue] [-latex] (h1) edge (3);
    \draw[MyRed] [-latex] (h2) edge (2);
    \draw[MyRed] [-latex] (h2) edge (3);
    \draw[MyRed] [-latex] (h2) edge (4);
    \draw[MyGreen] [-latex] (h3) edge (3);
    \draw[MyGreen] [-latex] (h3) edge (4);
    \draw[MyGreen] [-latex] (h3) edge (5);
    \draw[MyGreen] [-latex] (h3) edge (6);
\end{tikzpicture}
(c)
\begin{tikzpicture}[align=center]
    \node[fill=lightgray] (h1) at (-1,1) {$h_1$};
    \node[fill=lightgray] (h2) at (-0.5,-1) {$h_2$};
    \node[fill=lightgray] (h3) at (1,1) {$h_3$};
    
    \node[] (1) at (-2.5,0) {$v_1$};
    \node[] (2) at (-1.5,0) {$v_2$};
    \node[] (3) at (-0.5,0) {$v_3$};
    \node[] (4) at (0.5,0) {$v_4$};
    \node[] (5) at (1.5,0) {$v_5$};
    \node[] (6) at (2.5,0) {$v_6$};
    
    \draw[MyBlue] [-latex] (h1) edge (1);
    \draw[MyBlue] [-latex] (h1) edge (2);
    \draw[MyBlue] [-latex] (h1) edge (3);
    \draw[MyBlue] [-latex] (h1) edge (4);
    \draw[MyBlue] [-latex] (h1) edge (5);
    \draw[MyBlue] [-latex] (h1) edge (6);
    \draw[MyRed] [-latex] (h2) edge (2);
    \draw[MyRed] [-latex] (h2) edge (3);
    \draw[MyRed] [-latex] (h2) edge (4);
    \draw[MyGreen] [-latex] (h3) edge (3);
    \draw[MyGreen] [-latex] (h3) edge (4);
    \draw[MyGreen] [-latex] (h3) edge (6);
\end{tikzpicture}
\caption{Graphs with lower model dimension than number of parameters.}
\label{fig:dimension}
\end{figure}
\end{center}

To study the dimension of sparse factor analysis models, we first introduce necessary terminology. Let $C(V,2) := \left\{ \{v,w\}: v,w \in V, v \neq w \right\}$ be the set of $2$-pairs of $V$, i.e., the set of all subsets consisting of $2$ distinct nodes of $V$. We write $\jpa(\{u,v\}) = \{h \in \mathcal{H}: h \in \pa(u) \cap \pa(v)\}$ for the set of \emph{joint parents} of a pair $\{u,v\} \in C(V,2)$. For any latent node $h \in \mathcal{H}$, we let $C(V,2)_h = \left\{ \{v,w\} \in C(V,2): h \in \jpa(\{v,w\})\right\}$ be the collection of pairs of nodes that have $h$ as a joint parent.

For a matrix $\Sigma \in F(G)$, the parametrization of the entries $\sigma_{uv}$ depends on the joint parents of the pair $\{u,v\}$. In particular, for $\Omega = \text{diag}(\omega_{vv}) \in \mathbb{R}^p_{>0}$ and $\Lambda = (\lambda_{vh}) \in \mathbb{R}^D$, we have
\[
\sigma_{uv} = \begin{cases} 
\sum_{h \in \jpa(\{u,v\})} \lambda_{uh} \lambda_{vh} & \text{if } u \neq v ,\\
\omega_{uu} + \sum_{h \in \pa(u)} \lambda_{uh}^2 & \text{if } u = v,
\end{cases}
\]
where we use the convention that the empty sum is zero. Recall that the dimension of $F(G)=\im(\tau_G)$ is equal to the maximal rank of the Jacobian of $\tau_G$. Hence, we need to study the Jacobian matrix that has the form
\[
J = \begin{pNiceMatrix}[first-row,first-col]
    & \omega & \lambda  \\
u & I_p   & C  \\
\{u,v\} & 0   & B 
\end{pNiceMatrix},
\]
where the rows in the upper part correspond to the derivatives of $\sigma_{uu}$ and the rows in the lower part correspond to the derivatives of $\sigma_{uv}$ for $u \neq v$. In particular, the entries in the unit matrix $I_p$ on the upper left are given by
\[
\frac{\partial \sigma_{uu}}{\partial \omega_{vv}} = \begin{cases} 
1 & \textrm{if } u=v, \\
0 & \textrm{else}.
\end{cases}
\]
Thus, the rank of the Jacobian is equal to $p+\rk(B)$; recall that $p=|V|$. The entries of the matrix $B$ are given by
\begin{equation} \label{eq:entries-of-jacobian}
\frac{\partial \sigma_{uv}}{\partial \lambda_{zh}} = \begin{cases} 
\lambda_{vh} & \textrm{if } z=u \textrm{ and } h \in \jpa(\{u,v\}), \\
\lambda_{uh} & \textrm{if } z=v \textrm{ and } h \in \jpa(\{u,v\}), \\
0 & \textrm{else}.
\end{cases}
\end{equation}
Note that the rows of $B$ are indexed by 2-pairs $\{u,v\} \in C(V,2)$. A necessary condition for a model to have expected dimension is the crucial observation that, for each latent node $h$, there has to be a \emph{different} set of 2-pairs of children of $h$ that has same cardinality as the number of children of $h$. Otherwise, the dimension drops accordingly.   We formalize the concept of different 2-pairs, a.k.a rows of $B$, by considering  pairwise disjoint collections.

\begin{definition}
    Let $G=(V \cup \mathcal{H})$ be a factor analysis graph and let $\mathcal{A}=(A_h)_{h \in \mathcal{H}}$ be a collection of observed 2-pairs, that is, $A_h \subseteq C(V,2)$. We say that the collection $\mathcal{A}$ is \emph{valid} if 
    \begin{itemize}
        \item[(i)] $A_h \subseteq C(V,2)_h$ with cardinality $|A_h| \leq  |\ch(h)|$ for all $h \in \mathcal{H}$, and
        \item[(ii)] the collection is pairwise disjoint, i.e., $A_h \cap A_{\ell} = \emptyset$ for $h \neq \ell$.
    \end{itemize}
    Moreover, we say that $\sum_{h \in \mathcal{H}} |A_h|$ is the \emph{sum of cardinalities} of a valid collection.
\end{definition}

The next theorem gives an upper bound on the dimension. It is obtained by choosing a valid collection $\mathcal{A}=(A_h)_{h \in \mathcal{H}}$ such that the sum of cardinalities $\sum_{h \in \mathcal{H}} |A_h|$ is maximal. The upper bound holds for all sparse factor models, even if ZUTA is not satisfied.

\begin{theorem} \label{thm:dimension-general}
Let $G=(V \cup \mathcal{H}, D)$ be a factor analysis graph. Let $\mathcal{A}=(A_h)_{h \in \mathcal{H}}$ be a valid collection of $2$-pairs such that the sum of cardinalities $\sum_{h \in \mathcal{H}} |A_h|$ is maximal among all valid collections.   Then
\[
     \dim(F(G)) \leq |V| + \sum_{h \in \mathcal{H}} |A_h|.
\]
\end{theorem}

\begin{proof}
It is enough to show that $\rk(B) \leq \sum_{h \in \mathcal{H}} |A_h|$. Define $\lambda_{i} := (\lambda_{\ch(h_i),h_i}) \in \mathbb{R}^{|\ch(h_i)|}$ and $\mathcal{A}^{\complement} := C(V,2) \setminus \left(\bigcup_{h \in \mathcal{H}} A_{h}\right)$. 
Then the matrix $B$ can be written as
\begin{equation*}
B \; = \;\; \begin{pNiceArray}{ccc}[margin,parallelize-diags=false,first-row,first-col]
    & \lambda_{1} & \cdots & \lambda_{m} \\
A_{h_1} & B_{1,1} & \cdots & B_{1,m}\\
\Vdots & \vdots &  &  \vdots \\
A_{h_m}  & B_{m,1} & \cdots & B_{m,m}  \\
\hline
\mathcal{A}^{\complement} & B_{\mathcal{A}^{\complement},1} & \cdots & B_{\mathcal{A}^{\complement},m} 
\end{pNiceArray}.
\end{equation*}
The proof is structured as follows. We first show in Claim 1 and Claim 3 that some submatrices of $B$ are equal to zero. Claim 2 is an intermediate result we need to prove for Claim 3. Then, we restructure the matrix $B$ and show in Claim 4 that the rank of the matrix $B$ can not be larger than  $\sum_{h \in \mathcal{H}} |A_h|$.
Let $[m]:=\{1,\ldots,m\}$ and define the index sets $I^{(=)} = \{i \in [m]: |A_{h_i}|=|\ch(h_i)|\}$ and $I^{(<)} = [m] \setminus I^{(=)}$. \looseness=-1

\mbox{ } \\
\noindent
\textbf{Claim 1:} If $i \in I^{(<)}$, then $B_{\mathcal{A}^{\complement},i} = 0$. \\

Consider an index $i \in I^{(<)}$ and a row indexed by $\{u,z\} \in \mathcal{A}^{\complement}$. Observe that we must have $h_i \not\in \pa(u) \cap \pa(z)$. Otherwise we could have chosen $\widetilde{A}_{h_i} = A_{h_i} \cup \{\{u,z\}\}$ that has empty intersection with any $A_{h_j}$ for $j \neq i$. But this defines another valid collection $\widetilde{\mathcal{A}}=(A_{h_1}, \ldots, A_{h_{i-1}}, \widetilde{A}_{h_i}, A_{h_{i+1}}, \ldots, A_{h_m})$ such that the sum of cardinalities is greater by one, which contradicts the assumption on the maximality of $\mathcal{A}$. We conclude that the row in $B_{\mathcal{A}^{\complement},i}$ that is indexed by $\{u,z\}$ is equal to zero; recall \eqref{eq:entries-of-jacobian}. Since this holds for all rows $\{u,z\} \in \mathcal{A}^{\complement}$, we have that $B_{\mathcal{A}^{\complement},i}=0$, which proves the claim. \\

\noindent To state Claim 2 we define 
\[
    J^0 = \{i \in I^{(=)}: B_{\mathcal{A}^{\complement}, i} \neq 0\} = \{i \in I^{(=)}: h_i \in \jpa(R) \text{ for some } R \in \mathcal{A}^{\complement}\},
\]
and, for all $k \geq 1$, we define 
\[
    J^k = \{j \in I^{(=)}: \text{there is } i \in J^{k-1} \text{ such that } 
    h_j \in \text{jpa}(R) \text{ for some } R \in A_{h_i}\}. 
\]
Since $h_j \in \jpa(R)$ for all $R \in A_{h_j}$, we clearly have that $J^{k} \subseteq J^{k+1}$ for all $k \geq 0$. 

\mbox{ } \\
\noindent
\textbf{Claim 2:} Let  $R \in A_{h_j}$ for some $j \in J^k$, $k \geq 0$. Then $\{i \in [m]: h_i \in \text{jpa}(R)\} \subseteq I^{(=)}$. \\

We first assume $k=0$. Let $j \in J^0$ and $R \in A_{h_j}$, and suppose there is $h_l \in \jpa(R)$ such that $l \in I^{(<)}$. On the one hand, this means that $|A_{h_l}|<|\ch(h_l)|$. On the other hand, since $j \in J^0$, there has to be a pair $S \in \mathcal{A}^{\complement}$ such that $h_j \in \jpa(S)$. Therefore, we can define a collection $\widetilde{\mathcal{A}}=(\widetilde{A}_{h})_{h \in \mathcal{H}}$ such that $\widetilde{A}_{h_j} = (A_{h_j} \setminus \{R\}) \cup \{S\}$, $\widetilde{A}_{h_l} = A_{h_l} \cup \{R\}$ and $\widetilde{A}_{h}=A_{h}$ for all $h \not\in \{h_j,h_l\}$. Note that the collection $\widetilde{\mathcal{A}}$ is valid but $\sum_{h \in \mathcal{H}}|\widetilde{A}_h|=1+\sum_{h \in \mathcal{H}}|A_h|$. This is a contradiction to the maximality assumption on the collection $\mathcal{A}$ and we conclude that we must have $l \in I^{(=)}$. 

Now, let $k \geq 1$ and assume that $j \in J^k$ and $R \in A_{h_j}$. If $j \in J^0$, we are done. If $j \in J^k \setminus J^0$, suppose there is $h_l \in \jpa(R)$ such that $l \in I^{(<)}$. Once again, this means that $|A_{h_l}| < |\ch(h_l)|$. Now, we recursively choose integers $j_1, \ldots, j_n$ and corresponding subsets $S_{j_i} \in A_{h_{j_i}}$ as follows. First, since $j \in J^k$, there has to be a minimal integer $k_1 < k$ such that there is $j_1 \in J^{k_1}$ and $h_j \in \text{jpa}(S_{j_1})$ for some $S_{j_1} \in A_{h_{j_1}}$. Note that $j \not\in J^{k_1}$, since otherwise there would exist $\tilde{j}_1 \in J^{k_1-1}$ such that $h_j \in \jpa(\widetilde{S})$ for some $\widetilde{S} \in A_{h_{\tilde{j}_1}}$, which is a contradiction on the minimality assumption on $k_1$. 
Further,  define $k_{i+1}$ as the minimal integer $k_{i+1} < k_i$ such that there is $j_{i+1} \in J^{k_{i+1}}$ and $h_{j_i} \in \text{jpa}(S_{j_{i+1}})$  for some $S_{j_{i+1}} \in A_{h_{j_{i+1}}}$. We stop this procedure as soon as we arrived at some $n \geq 1$ such that $k_{n}=0$. It can be seen as before that $j_i \not\in J^{k_{i+1}}$ for all $i=1, \ldots, n$. Hence, the integers $j, j_1, \ldots, j_{n}$ are pairwise different by construction, which also implies that the pairs $R, S_{j_1}, \ldots, S_{j_n}$ are pairwise different.
Moreover, since $j_{n} \in J^0$, there has to be a pair $S \in \mathcal{A}^{\complement}$ such that $h_{j_n} \in \jpa(S)$. Now, we define a collection $\widetilde{\mathcal{A}}=(\widetilde{A}_{h})_{h \in \mathcal{H}}$ as follows:
\begin{align*}
\widetilde{A}_{h_{j_{n}}} &= (A_{h_{j_{n}}} \setminus \{S_{j_n}\}) \cup \{S\}, &
\widetilde{A}_{h_{j_{i}}} &= (A_{h_{j_{i}}} \setminus \{S_{j_i}\}) \cup \{S_{j_{i+1}}\} \ \text{ for } i = 1, \ldots n-1, \\
\widetilde{A}_{h_j} &=(A_{h_{j}} \setminus \{R\}) \cup \{S_{j_1}\}, &
\widetilde{A}_{h_l} &= A_{h_l} \cup \{R\},
\end{align*}
and $\widetilde{A}_{h}=A_{h}$ for all $h \in \mathcal{H}$ that do not appear above. Since the pairs $R, S_{j_1}, \ldots, S_{j_n}$ and $S$ are pairwise different, the collection is valid. However, we have that $\sum_{h \in \mathcal{H}}|\widetilde{A}_h|=1+\sum_{h \in \mathcal{H}}|A_h|$, which is a contradiction to the maximality assumption on the collection $\mathcal{A}$. We conclude that we must have $l \in I^{(=)}$, which proves the claim. \\

Now, observe that there must exist a $k^{\ast} \geq 0$ such that the sequence $J^0 \subseteq J^1 \subseteq \ldots$ stabilizes, that is, $J^{k^{\ast}} = J^{k^{\ast}+1} = \ldots$. This is true since  $J^{k-1} \subseteq J^k$ and $J^k \subseteq I^{(=)}$ for all $k \geq 1$. Define $\overline{J} := J^{k^{\ast}}$. 

\mbox{ } \\
\noindent
\textbf{Claim 3:} $B_{j,i}=0$ for all $j \in \overline{J}, i \in [m] \setminus \overline{J}$. \\

Let $j \in \overline{J}$ and $i \in [m] \setminus \overline{J}$ be two indices and consider a row in $B_{i,j}$ indexed by a pair $R \in A_{h_j}$. To show that this row of $B_{j,i}$ is zero, it is enough to show that $\{i \in [m]: h_i \in \text{jpa}(R)\} \subseteq \overline{J}$; recall Equation~\eqref{eq:entries-of-jacobian}. By Claim 2, we have that $\{i \in [m]: h_i \in \text{jpa}(R)\} \subseteq I^{(=)}$. Now, assume that there is $h_l \in \jpa(R)$ such that $l \not\in \overline{J}$. By definition, this means that $l \in J^{k^{\ast}+1}$. But this is a contradiction since $\overline{J}=J^{k^{\ast}+1}$. We conclude that we must have $l \in \overline{J}.$

\mbox{ } \\
\noindent
\textbf{Claim 4:} The rank of the matrix $B$ cannot exceed $\sum_{h \in \mathcal{H}} |A_h|$. \\

Without loss of generality, $\overline{J}=[k]$ for some positive integer $k \leq m$. 
Then, by Claims 1-3, the matrix $B$ has the form
\begin{equation*}
B = 
\begin{pNiceArray}{ccc|ccc}[margin,parallelize-diags=false,first-row,first-col]
& \lambda_{1} & \cdots & \lambda_{k} & \lambda_{k+1}& \cdots & \lambda_{m} \\
A_{h_1} & B_{1,1} & \cdots &  B_{1,k} & 0 & \cdots & 0 \\
\Vdots & \vdots &  & \vdots & \vdots & & \vdots\\
A_{h_k} & B_{k,1} & \cdots & B_{k,k} & & &\\
\mathcal{A}^{\complement}& B_{\mathcal{A}^{\complement},1} & \cdots & B_{\mathcal{A}^{\complement},k}  & 0 & \cdots 
& 0 \\
\hline
A_{h_{k+1}}& B_{k+1,1} & \cdots & B_{k+1,k} & B_{k+1,k+1} & \cdots & B_{k+1,m}\\
\Vdots & \vdots & & \vdots &\vdots &  & \vdots \\
A_{h_m}& B_{m,1} & \cdots & B_{m,k} & B_{m,k+1}& \cdots& B_{m,m}\\
\end{pNiceArray}.
\end{equation*}
\mbox{ } \\

The rank of this matrix is smaller or equal to the sum of the minimum of the number of rows and columns of the upper left block plus the minimum of the number of rows and columns of the lower right block. The minimum of the number of rows and columns of the upper left block is given by $\min\{\sum_{h \in \overline{J}}|A_{h}| + |\mathcal{A}^{\complement}|, \sum_{h \in \overline{J}} |\ch(h)|\}$. Since $\sum_{h \in \overline{J}} |\ch(h)| = \sum_{h \in \overline{J}}|A_{h}|$, this minimum is equal to $\sum_{h \in \overline{J}}|A_{h}|$. On the other hand, the minimum of the number of rows and columns of the lower right block is given by $\min\{\sum_{h \in \overline{J}^{\complement}}|A_{h}|,  \sum_{h \in \overline{J}^{\complement}} |\ch(h)| \}$. Since $\sum_{h \in \overline{J}^{\complement}}|A_{h}| \leq \sum_{h \in \overline{J}^{\complement}}|\ch(h)|$, this minimum is given by  $\sum_{h \in \overline{J}^{\complement}}|A_{h}|$. Thus, the rank of the matrix $B$ cannot be larger than 
\[
    \sum_{h \in \overline{J}}|A_{h}| + \sum_{h \in \overline{J}^{\complement}}|A_{h}| = \sum_{h \in \mathcal{H}}|A_{h}|.
    \qedhere
\]
\end{proof}

\begin{example} \label{ex:upper-bound}
    Consider the graph in Figure~\ref{fig:dimension} (b). Then we have
\begin{align*}
C(V,2)_{h_1} &= \{\{v_1,v_2\}, \{v_1,v_3\}, \{v_2,v_3\}\}, \\
C(V,2)_{h_2} &= \{\{v_2,v_3\}, \{v_2,v_4\}, \{v_3,v_4\}\},\\
C(V,2)_{h_3} &= \{\{v_3,v_4\}, \{v_3,v_5\}, \{v_3,v_6\}, \{v_4,v_5\}, \{v_4,v_6\}, \{v_5,v_6\}\}.
\end{align*}
To obtain an upper bound for the dimension, we want to choose the subsets $A_{h_i} \subseteq C(V,2)_{h_i}$ with cardinality as large as possible but not larger than the number of children. However, to obtain a valid, i.e.\ pairwise disjoint, collection we have to choose either $|A_{h_1}|=2$ or $|A_{h_2}|=2$. If both $|A_{h_1}|=3$ and $|A_{h_2}|=3$, then we must have that $\{v_2,v_3\} \in A_{h_1} \cap A_{h_2}$, i.e., the collection is not pairwise disjoint. On the other hand, we can choose $A_{h_3}$ with cardinality at most 4, e.g.\ $\{\{v_3,v_5\}, \{v_3,v_6\}, \{v_4,v_6\}, \{v_5,v_6\}\}$ that does not intersect with any of $A_{h_1}$ and $A_{h_2}$. Thus, any pairwise disjoint collection $\mathcal{A}=(A_{h_1}, A_{h_2}, A_{h_3})$ with $|A_{h_i}| \leq |\ch(h_i)|$ has a maximal sum of cardinalities equal to $2+3+4=9$. Applying the upper bound in Theorem~\ref{thm:dimension-general}, we obtain that $\dim(F(G))) \leq  6 + 9 = 15$ which is strictly less than the expected dimension 16.
\end{example}

While Theorem \ref{thm:dimension-general} holds for any sparse factor analysis graph, also for graphs that do not satisfy ZUTA, to obtain a lower bound on the dimension, we consider more refined collections of 2-pairs that require ZUTA to be satisfied. If ZUTA is satisfied, we can assume that the latent nodes are labeled as  $\mathcal{H}=\{h_1, \ldots, h_m\}$ and the observed nodes are labeled as $V=\{v_1, \ldots, v_p\}$ such that $v_i \in \ch(h_i)$ and $v_i \not\in \bigcup_{j > i} \ch(h_j)$ for all $i=1, \ldots, m$. 

\begin{definition}
    Suppose that ZUTA is satisfied. A valid collection $\mathcal{A}=(A_h)_{h \in \mathcal{H}}$ of $2$-pairs is \emph{ZUTA-compliant} if  $\{v_i,w\} \in A_{h_i}$ for all $w \in \ch(h_i) \setminus \{v_i\}$ and for all $i \in [m]$.
\end{definition}

Note that a valid, ZUTA-compliant collection always exists for a factor analysis graph that satisfies ZUTA. Indeed, one may just choose $A_{h_i}=\{\{v_i,w\}: w \in \ch(h_i) \setminus \{v_i\}\}$. In this collection, the cardinality of each set of 2-pairs $A_{h_i}$ is equal to $|\ch(h_i)|-1$. However, there might exist other valid, ZUTA-compliant collections where the components $A_{h_i}$ potentially contain one more 2-pair, that is, $A_{h_i}$ might be chosen such that its cardinality is equal to $|\ch(h_i)|$. Each of these ZUTA-compliant collections gives a lower bound on the dimension as we prove in the next theorem.

\begin{theorem} \label{thm:dim-lower-bound}
    Let $G=(V \cup \mathcal{H}, D)$ be a factor analysis graph. Suppose that ZUTA is satisfied and let $\mathcal{A}=(A_h)_{h \in \mathcal{H}}$ be a valid collection that is ZUTA-compliant. Then,
    \begin{equation} \label{eq:lower-bound}
        \dim(F(G)) \geq |V| + \sum_{h \in \mathcal{H}} |A_h|.
    \end{equation}
\end{theorem}
\begin{proof}
It suffices to show that, for generic parameter choices,  the rank of $B$ is larger or equal to $r = \sum_{h \in \mathcal{H}} |A_h|$. Let $[m]:=\{1,\ldots,m\}$ and define the index sets $I^{(=)} = \{i \in [m]: |A_{h_i}|=|\ch(h_i)|\}$ and $I^{(<)} = [m] \setminus I^{(=)}$ as in the proof of Theorem~\ref{thm:dimension-general}. Consider the sets $C_i := \{\{v_i, w\}: w \in  \ch(h_i) \setminus \{v_i\}\} \subseteq C(V,2)_{h_i}$ that have a cardinality of at most $|\ch(h_i)|-1$ and are pairwise disjoint. By definition, the collection $\mathcal{A}$ is given by
\begin{equation} \label{eq:def-collection}
    A_{h_i} = \begin{cases}
        C_i  \cup \{S_i\} & \text{if } i \in I^{(=)}, \\
        C_i & \text{if } i \in I^{(<)},
    \end{cases}
\end{equation}
where $S_i \in C(V,2)_h \setminus C_i$. 
Now, let $\ch(h_i)^{-} = \ch(h_i) \setminus \{v_i\}$ and $\lambda_{i}^{-} = \lambda_{\ch(h_i)^{-},h_i} \in \mathbb{R}^{|\ch(h_i)|-1}$. Moreover, we write $\mathcal{S} = \{S_i: i \in I^{(=)}\}$ and $\mathcal{A}^{\complement} = C(V,2) \setminus \left(\bigcup_{h \in \mathcal{H}} A_{h}\right)$. To see that the matrix $B$ has generically has at least rank $r$, we arrange it as 
\begin{equation} \label{eq:matrix-B}
B \; = \;\; \begin{pNiceArray}{ccc|ccc}[margin,parallelize-diags=false,first-row,first-col]
    & \lambda_{v_1, h_1}  & \cdots & \lambda_{v_m, h_m} & \lambda_{1}^{-} & \cdots & \lambda_{m}^{-} \\
C_1 & \lambda_{1}^{-} & &  & B_{1,1} & & \\
\Vdots & & \ddots & & \vdots & \ddots & \\
C_m  & & & \lambda_{m}^{-} & B_{m,1} & \cdots & B_{m,m}  \\
\hline
\mathcal{S} & & & & B_{\mathcal{S},1} & \cdots & B_{\mathcal{S},m} \\
\mathcal{A}^{\complement} & & & & B_{\mathcal{A}^{\complement},1} & \cdots & B_{\mathcal{A}^{\complement},m} \\
\end{pNiceArray},
\end{equation}
where void entries are zero; recall ZUTA and Equation~\eqref{eq:entries-of-jacobian}. Now, we choose a specific matrix $\Lambda^0 = (\lambda_{v_i, h_j}^0) \in \mathbb{R}^D$ such that the rank of $B^0 = B(\Lambda^0)$ is at least $r = \sum_{h \in \mathcal{H}} |A_h| = \sum_{i \in [m]} |C_i| + |\mathcal{S}|$. The existence of such a matrix implies that the rank of $B^0$ is at least $r$ for a generic choice of $\Lambda$. We choose the entries of $\Lambda^0$ as follows. For all $i \in [m]$, we set $\lambda_{v_i, h_i}^0 = 1$. Since the submatrices $B_{i,i} = \lambda_{v_i,h_i} I_{|\ch(h_i)|-1}$ are diagonal, this implies that the upper right block of $B$ is of full rank $ \sum_{h \in \mathcal{H}} (|\ch(h)|-1) = \sum_{i \in [m]} |C_i|$. The remaining non-zero entries of $\Lambda^0$ are determined as follows. For any row $S_i = \{u_i,w_i\} \in \mathcal{S}$, we set exactly two entries equal to one, namely $\lambda_{u_i,h_i}^0$ and $\lambda_{w_i,h_i}^0$. All other entries of $\Lambda^0$ remain zero. Since the matrix $\Lambda^0$ has entries in $\{0, 1\}$, the same holds for the matrix $B^0$; recall Equation~\eqref{eq:entries-of-jacobian} again. 

To show that the rank of $B^0$ is at least $r = \sum_{i \in [m]} |C_i| + |\mathcal{S}|$, we proceed by row reduction. First, we consider blocks $B_{i, j}^0$ with $j<i$ of the upper right of $B^0$, and show that we can eliminate all nonzero entries in this block by subtracting certain rows indexed by elements in $C_j$, without creating any additional nonzero entries. We show this intermediate result in Claim 1. In a second step, we eliminate all nonzero entries in $B_{\mathcal{S},i}$ for all $i \in [m]$ by subtracting specific rows in $C_1 \cup \cdots \cup C_m$. This elimination will create fill-ins in the submatrix of $B^0$ that is indexed by the rows in $\mathcal{S}$ and the columns of the left-hand side in~\eqref{eq:matrix-B}. However, the fill-ins are precisely such that the submatrix contains a multiple of a permutation matrix of size $|\mathcal{S}|$, which implies that it has full row rank. This fact will be shown in Claim 2. Both claims together imply the statement of the theorem. 
We also illustrate the row reduction in Example~\ref{ex:row-reduction}.

\mbox{ } \\
\noindent
\textbf{Claim 1:} By row reduction,  the upper right block of $B^0$ can be transformed into a diagonal matrix of size $\sum_{i \in [m]} |C_i|$, while no fill-in occurs in the upper left block of $B^0$. \\

Consider a block $B_{i, j}^0$ with $j<i$ of the upper right block of $B^0$. Fix any row in this block indexed by $\{v_i,z\} \in C_i$. By Equation~\eqref{eq:entries-of-jacobian}, this row is zero if $h_j$ is not a joint parent of $\{v_i,z\}$. Now, consider the case where $h_j \in \jpa(\{v_i,z\})$. Then, the row may contain two potential nonzero entries given by $\lambda^0_{v_i,h_j}$ and $\lambda^0_{z,h_j}$ (occuring at the entries with column indices $\lambda_{z,h_j}$ and $\lambda_{v_i,h_j}$). Since $v_i$ does not appear in any pair in $\mathcal{S}$, the entry $\lambda^0_{v_i,h_j}$ must be zero. If $\lambda_{z,h_j}^0=1$, we can eliminate this entry by subtracting the row indexed by $\{v_j,v_i\}\in C_j$. The relevant submatrix of $B^0$ is given by 
\[
\begin{pNiceArray}{c|cc}[margin,parallelize-diags=false,first-row,first-col]
    & \lambda_{v_j, h_j}  & \lambda_{v_i, h_j}  & \lambda_{z, h_j}  \\
\{v_j,v_i\} & \lambda^0_{v_i,h_j} & \lambda^0_{v_j,h_j} & 0 \\
\{v_i,z\} & 0 & \lambda^0_{z,h_j} &  \lambda^0_{v_i, h_j}\\
\end{pNiceArray}
=
\begin{pNiceArray}{c|cc}[margin,parallelize-diags=false,first-row,first-col]
    & \lambda_{v_j, h_j}  & \lambda_{v_i, h_j}  & \lambda_{z, h_j} \\
\{v_j,v_i\} & 0 & 1 & 0 \\
\{v_i,z\} & 0 & 1 & 0\\
\end{pNiceArray}.
\]
No fill-in occurs in the entry indexed by $\{v_i,z\}$ and $\lambda_{v_j, h_j}$ since this would only happen if both $\lambda_{v_i,h_j}^0$ and $\lambda_{z,h_j}^0$ are equal to one. Hence, 
no fill-in occurs in the upper left block of $B^0$.  As claimed, we conclude that the upper right block of $B^0$ can be transformed into a diagonal matrix, while no fill-in occurs in the upper left block of $B^0$.

\mbox{ } \\
\noindent
\textbf{Claim 2:} By row reduction, the submatrix of $B^0$ that consists of the rows indexed by $\mathcal{S}$ can be transformed such that the left block with columns index by $\lambda_{v_1, h_1}, \ldots, \lambda_{v_m, h_m}$ is of full row rank, and the right block with columns indexed by $\lambda_1^{-}, \ldots, \lambda_m^{-}$ is zero.\\

For all $i \in [m]$, we eliminate the nonzero entries in $B_{\mathcal{S},i}$  by subtracting rows from the upper half indexed by elements in $C_1 \cup \cdots \cup C_m$. 
Fix any row in $B_{\mathcal{S},i}$ indexed by $S_j = \{y,z\} \in \mathcal{S}$. Observe that the row $B_{S_j,i}$ in the submatrix $B_{\mathcal{S},i}$ is zero if $h_i$ is not a joint parent of $S_j$, recall Equation~\eqref{eq:entries-of-jacobian}. Now, consider a joint parent $h_i \in \jpa(S_j)$. Then, the row $B_{S_j,i}$ in the submatrix $B_{\mathcal{S},i}$ may contain two potential nonzero entries given by $\lambda^0_{y,h_i}$ and $\lambda^0_{z,h_i}$ (occuring at the entries with column indices $\lambda_{z,h_i}$ and $\lambda_{y,h_i}$). If $i=j$, then both $\lambda_{y,h_j}^0$ and $\lambda_{z,h_j}^0$ are equal to one. We eliminate the two entries by subtracting the rows indexed by $\{v_j,y\} \in C_j$ and $\{v_j,z\} \in C_j$. The relevant submatrix of $B^0$ is given by
\[
\begin{pNiceArray}{c|cc}[margin,parallelize-diags=false,first-row,first-col]
    & \lambda_{v_j, h_j}  & \lambda_{y, h_j} & \lambda_{z, h_j}  \\
\{v_j,y\} & \lambda^0_{y,h_j} & \lambda^0_{v_j,h_j} & 0\\
\{v_j,z\} & \lambda^0_{z,h_j} & 0 & \lambda^0_{v_j,h_j} \\
\hline
\{y,z\} & 0 &  \lambda^0_{z,h_j} &  \lambda^0_{y,h_j}
\end{pNiceArray}
=
\begin{pNiceArray}{c|cc}[margin,parallelize-diags=false,first-row,first-col]
    & \lambda_{v_j, h_j}  & \lambda_{y, h_j} & \lambda_{z, h_j}  \\
\{v_j,y\} & 1 & 1 & 0\\
\{v_j,z\} & 1 & 0 & 1 \\
\hline
\{y,z\} & 0 &  1 &  1
\end{pNiceArray}.
\]
Fill-in occurs in the entry indexed by $\{y,z\}$ and $\lambda_{v_j, h_j}$ and is equal  to $-2$, i.e, after elimination this submatrix of $B^0$ if given  by
\[
\begin{pNiceArray}{c|cc}[margin,parallelize-diags=false,first-row,first-col]
    & \lambda_{v_j, h_j}  & \lambda_{y, h_j} & \lambda_{z, h_j}  \\
\{v_j,y\} & 1 & 1 & 0\\
\{v_j,z\} & 1 & 0 & 1 \\
\hline
\{y,z\} & -2 &  0 &  0
\end{pNiceArray}.   
\]
Now, consider a joint parent $h_i \in \jpa(S_j)$ with $i \neq j$. 
In this case, at least one of $\lambda_{y,h_i}^0$ and $\lambda_{z,h_i}^0$ is  zero. Since we are done if both are zero, we can assume w.l.o.g.~that $\lambda_{z,h_i}^0=1$ and $\lambda_{y,h_i}^0=0$. We  eliminate this entry by subtracting the row indexed by $\{v_i,y\} \in C_i$. The relevant submatrix of $B^0$ is given by
\[
\begin{pNiceArray}{c|cc}[margin,parallelize-diags=false,first-row,first-col]
    & \lambda_{v_i, h_i}  & \lambda_{y, h_i} & \lambda_{z, h_i}  \\
\{v_i,y\} & \lambda^0_{y,h_i} & \lambda^0_{v_i,h_i} & 0 \\
\hline
\{y,z\} & 0 &  \lambda^0_{z,h_i} & \lambda^0_{y,h_i}
\end{pNiceArray}
=
\begin{pNiceArray}{c|cc}[margin,parallelize-diags=false,first-row,first-col]
    & \lambda_{v_i, h_i}  & \lambda_{y, h_i}  & \lambda_{z, h_i}  \\
\{v_i,y\} & 0 & 1 & 0\\
\hline
\{y,z\} & 0 &  1 & 0
\end{pNiceArray}.
\]
No fill-in occurs in the entry indexed by $\{y,z\}$ and $\lambda_{v_i, h_i}$ since this would only happen if both $\lambda_{y,h_i}^0$ and $\lambda_{z,h_i}^0$ are equal to one. 

To summarize, we have shown that, after elimination, an entry of $B^0$ with row indexed by $S_j \in \mathcal{S}$ and column indexed by $\lambda_{v_i,h_i}$ is equal to $-2$ if $i=j$ and $0$ else. Hence, after elimination, the submatrix of $B^0$ that is indexed by the rows in $\mathcal{S}$ and the columns of the left-hand side in~\eqref{eq:matrix-B} contains a permutation of $-2 \, I_{|\mathcal{S}|}$ and is therefore of full row rank. On the other hand, after elimination, the submatrix of $B^0$ that is indexed by the rows in $\mathcal{S}$ and the columns of the right-hand side in in~\eqref{eq:matrix-B} is zero, i.e., we have shown Claim 2. \\

We conclude the proof by noting that Claim 1 and Claim 2 directly imply that the rank of the matrix $B^0$ is at least $\sum_{i \in [m]} |C_i| + |\mathcal{S}| = r$, as we have claimed.
\end{proof}

\begin{figure}[t]
\begin{tikzpicture}[align=center]
    \node[fill=lightgray] (h1) at (-1.5,1) {$h_1$};
    \node[fill=lightgray] (h2) at (-0.5,-1) {$h_2$};
    \node[fill=lightgray] (h3) at (1,1) {$h_3$};
    
    \node[] (1) at (-2.5,0) {$v_1$};
    \node[] (2) at (-1.5,0) {$v_2$};
    \node[] (3) at (2.5,0) {$v_6$};
    \node[] (4) at (0.5,0) {$v_4$};
    \node[] (5) at (1.5,0) {$v_5$};
    \node[] (6) at (-0.5,0) {$v_3$};
    
    \draw[MyBlue] [-latex] (h1) edge (1);
    \draw[MyBlue] [-latex] (h1) edge (2);
    \draw[MyBlue] [-latex] (h1) edge (3);
    \draw[MyRed] [-latex] (h2) edge (2);
    \draw[MyRed] [-latex] (h2) edge (3);
    \draw[MyRed] [-latex] (h2) edge (4);
    \draw[MyGreen] [-latex] (h3) edge (3);
    \draw[MyGreen] [-latex] (h3) edge (4);
    \draw[MyGreen] [-latex] (h3) edge (5);
    \draw[MyGreen] [-latex] (h3) edge (6);
\end{tikzpicture}
\caption{Same graph as in Figure~\ref{fig:dimension} (b) with swapped labels $v_3$ and $v_6$.}
\label{fig:swapped-labels}
\end{figure}

\begin{example} \label{ex:row-reduction}
In this example, we illustrate the row reduction in the proof of Theorem~\ref{thm:dim-lower-bound}. Consider the graph in Figure~\ref{fig:swapped-labels} and the ZUTA-compliant, valid collection
\begin{align*}
A_{h_1} &= \{\{v_1,v_2\}, \{v_1,v_6\}\}, \\
A_{h_2} &= \{\{v_2,v_4\}, \{v_2,v_6\}, \{v_4,v_6\}\}, \\
A_{h_3} &= \{\{v_3,v_4\}, \{v_3,v_5\}, \{v_3,v_6\}, \{v_5,v_6\}\}.
\end{align*}
It follows that
\begin{align*}
C_1 &= \{\{v_1,v_2\}, \{v_1,v_6\}\}, \qquad
C_2 = \{\{v_2,v_4\}, \{v_2,v_6\}\}, \qquad
C_3 = \{\{v_3,v_4\}, \{v_3,v_5\}, \{v_3,v_6\}\},
\end{align*}
and that $\mathcal{S}=\{S_2,S_3\}$ with $S_2 = \{v_4,v_6\}$ and $S_3 = \{v_5,v_6\}$. Now, the matrix $B$ in~\eqref{eq:matrix-B} is given by
\begin{equation*} 
B = \; \begin{pNiceArray}{ccc|ccccccc}[margin,parallelize-diags=false,first-row,first-col]
    & \lambda_{v_1, h_1}  & \lambda_{v_2, h_2} & \lambda_{v_3, h_3} & \lambda_{v_2,h_1} & \lambda_{v_6,h_1} & \lambda_{v_4,h_2} & \lambda_{v_6,h_2} & \lambda_{v_4,h_3} &\lambda_{v_5,h_3} & \lambda_{v_6,h_3}\\
\{v_1,v_2\} & \lambda_{v_2,h_1} & & & \lambda_{v_1,h_1} & & & & & & \\
\{v_1,v_6\} & \lambda_{v_6,h_1} & & & & \lambda_{v_1,h_1} & & & & & \\
\hdashline
\{v_2,v_4\} & & \lambda_{v_4,h_2} & & & & \lambda_{v_2,h_2} & & & & \\
\{v_2,v_6\} & & \lambda_{v_6,h_2} & &  \lambda_{v_6,h_1} &  \lambda_{v_2,h_1} & & \lambda_{v_2,h_2} & & & \\
\hdashline
\{v_3,v_4\} & & & \lambda_{v_4,h_3} & & & & & \lambda_{v_3,h_3} & & \\
\{v_3,v_5\} & & & \lambda_{v_5,h_3} & & & & & & \lambda_{v_3,h_3} & \\
\{v_3,v_6\} & & & \lambda_{v_6,h_3} & & & & & & & \lambda_{v_3,h_3} \\
\hline
\{v_4,v_6\} & & & & & & \lambda_{v_6,h_2} & \lambda_{v_4,h_2} & \lambda_{v_6,h_3} & & \lambda_{v_4,h_3} \\
\{v_5,v_6\} & & & & & & & & & \lambda_{v_6,h_3} & \lambda_{v_5,h_3} \\
\hdashline
\{v_4,v_5\} & & & & & & & & \lambda_{v_5,h_3} & \lambda_{v_4,h_3} &  \\
\end{pNiceArray},
\end{equation*}
where void entries are zero. The matrix $B^0$ is constructed by setting the parameters $\lambda^0_{v_1,h_1}$, $\lambda^0_{v_2,h_2}$, $\lambda^0_{v_3,h_3}$, $\lambda^0_{v_4,h_2}$, $\lambda^0_{v_6,h_2}$, $\lambda^0_{v_5,h_3}$, and $\lambda^0_{v_6,h_3}$ to one, and all other parameters to zero. Hence, $B^0$ is given by \looseness=-1
\begin{equation*} 
B^0 = \; \begin{pNiceArray}{ccc|ccccccc}[margin,parallelize-diags=false,first-row,first-col]
    & \lambda_{v_1, h_1}  & \lambda_{v_2, h_2} & \lambda_{v_3, h_3} & \lambda_{v_2,h_1} & \lambda_{v_6,h_1} & \lambda_{v_4,h_2} & \lambda_{v_6,h_2} & \lambda_{v_4,h_3} &\lambda_{v_5,h_3} & \lambda_{v_6,h_3}\\
\{v_1,v_2\} & & & & 1 & & & & & & \\
\{v_1,v_6\} & & & & & 1 & & & & & \\
\hdashline
\{v_2,v_4\} & & 1 & & & & 1 & & & & \\
\{v_2,v_6\} & & 1 & & & & & 1 & & & \\
\hdashline
\{v_3,v_4\} & & & & & & & & 1 & & \\
\{v_3,v_5\} & & & 1 & & & & & & 1 & \\
\{v_3,v_6\} & & & 1 & & & & & & & 1 \\
\hline
\{v_4,v_6\} & & & & & & 1 & 1 & 1 & & \\
\{v_5,v_6\} & & & & & & & & & 1 & 1 \\
\hdashline
\{v_4,v_5\} & & & & & & & & 1 & &  \\
\end{pNiceArray}.
\end{equation*}
In this example, the statement of Claim 1 is already satisfied since the upper right block is a diagonal matrix. After eliminating the ones in the rows indexed by $\{v_4,v_6\}$ and $\{v_5,v_6\}$, the matrix is given by \looseness=-1
\begin{equation*} 
\begin{pNiceArray}{ccc|ccccccc}[margin,parallelize-diags=false,first-row,first-col]
    & \lambda_{v_1, h_1}  & \lambda_{v_2, h_2} & \lambda_{v_3, h_3} & \lambda_{v_2,h_1} & \lambda_{v_6,h_1} & \lambda_{v_4,h_2} & \lambda_{v_6,h_2} & \lambda_{v_4,h_3} &\lambda_{v_5,h_3} & \lambda_{v_6,h_3}\\
\{v_1,v_2\} & & & & 1 & & & & & & \\
\{v_1,v_6\} & & & & & 1 & & & & & \\
\hdashline
\{v_2,v_4\} & & 1 & & & & 1 & & & & \\
\{v_2,v_6\} & & 1 & & & & & 1 & & & \\
\hdashline
\{v_3,v_4\} & & & & & & & & 1 & & \\
\{v_3,v_5\} & & & 1 & & & & & & 1 & \\
\{v_3,v_6\} & & & 1 & & & & & & & 1 \\
\hline
\{v_4,v_6\} & & -2 & & & & & & & & \\
\{v_5,v_6\} & & & -2 & & & & & & & \\
\hdashline
\{v_4,v_5\} & & & & & & & & 1 & &  \\
\end{pNiceArray},
\end{equation*}
and we conclude that Claim 2 is satisfied. It is also easy to see that this matrix has at least rank $9$, i.e., the submatrix consisting of all but the last row has full row rank.
\end{example}

There might be several relabelings of observed and latent nodes such that ZUTA is satisfied. For each of these relabelings, one might potentially obtain a different lower bound in Theorem \ref{thm:dim-lower-bound}. Thus the best lower bound is obtained by maximizing the sum of cardinalities over all ZUTA-compliant, valid collections and over all relabelings such that ZUTA is satisfied. If the maximized lower bound from Theorem \ref{thm:dim-lower-bound} coincides with the upper bound in Theorem \ref{thm:dimension-general}, we obtain a formula for the dimension.

\begin{corollary} \label{cor:dimension-formula}
     Let $G=(V \cup \mathcal{H}, D)$ be a factor analysis graph and suppose that ZUTA is satisfied. If there is a ZUTA-compliant, valid collection $\mathcal{A}=(A_h)_{h \in \mathcal{H}}$ that has maximal sum of cardinalities $\sum_{h \in \mathcal{H}}|A_h|$ among all valid collections (which are not necessarily ZUTA-compliant), then 
     \[
        \dim(F(G)) = |V| + \sum_{h \in \mathcal{H}} |A_h|.
     \]
\end{corollary}

\begin{example}
Consider the graph in Figure~\ref{fig:swapped-labels} and the ZUTA-compliant, valid collection
\begin{align*}
A_{h_1} &= \{\{v_1,v_2\}, \{v_1,v_6\}\}, \\
A_{h_2} &= \{\{v_2,v_4\}, \{v_2,v_6\}, \{v_4,v_6\}\}, \\
A_{h_3} &= \{\{v_3,v_4\}, \{v_3,v_5\}, \{v_3,v_6\}, \{v_5,v_6\}\}.
\end{align*}
Observe that the graph in Figure~\ref{fig:swapped-labels} is equivalent to the  graph in Figure~\ref{fig:dimension} (b) when swapping the labels of $v_3$ and $v_6$. Hence, we have already seen in Example~\ref{ex:upper-bound} that this collection has maximal sum of cardinalities among all valid collections. It follows by Corollary~\ref{cor:dimension-formula} that $\dim(F(G))=6+9=15$.
\end{example}

If there is one pure child per latent node, the dimension formula from Corollary \ref{cor:dimension-formula} always holds.

\begin{corollary} \label{cor:1-pure-children}
     Let $G=(V \cup \mathcal{H}, D)$ be a factor analysis graph. Suppose that for every latent node $h \in \mathcal{H}$,  
     there is an observed node $v \in V$ such that $\pa(v)=\{h\}$. Let $\mathcal{A}=(A_h)_{h \in \mathcal{H}}$ be a valid collection that has maximal sum of cardinalities $\sum_{h \in \mathcal{H}}|A_h|$ among all valid collections. Then 
     \[
        \dim(F(G)) = |V| + \sum_{h \in \mathcal{H}} |A_h|.
     \]
\end{corollary}
\begin{proof}
Let $\mathcal{H}=\{h_1, \ldots, h_m\}$ and relabel the observed nodes $V=\{v_1, \ldots, v_p\}$ such that $v_i$ is a pure child of $h_i$, i.e., $\pa(v_i)=\{h_i\}$. To show the claim, it is enough by Corollary~\ref{cor:dimension-formula} to define another collection $\widetilde{\mathcal{A}}=(\widetilde{A}_h)_{h \in \mathcal{H}}$ that is also valid and has the same sum of cardinalities as $\mathcal{A}$, but is additionally ZUTA-compliant.

As in the proof of Theorem \ref{thm:dim-lower-bound}, define the index sets $I^{(=)} = \{i \in [m]: |A_{h_i}|=|\ch(h_i)|\}$ and $I^{(<)} = [m] \setminus I^{(=)}$. Consider the sets $C_i := \{\{v_i, w\}: w \in  \ch(h_i) \setminus \{v_i\}\} \subseteq C(V,2)_{h_i}$ that have a cardinality of at most $|\ch(h_i)|-1$. Moreover, they are pairwise disjoint since $h_i$ is the only parent of  $v_i$ with the given labeling.  
If $i \in I^{(=)}$, observe that the intersection $\{\{u,w\}: u,w \in \ch(h_i) \setminus \{v_i\} \} \cap A_{h_i}$ has to be nonempty. For any pair $\{u,w\}$ in this intersection, it must hold that neither $u$ nor $w$ is equal to $v_j$ for all $j \in [m] \setminus \{i\}$, since $v_j$ is a pure child of $h_j$. Hence, the pair $\{u,w\}$ is not contained in any $C_j$. Now, we choose a pair $S_i =\{u_i, w_i\}$ from the intersection $\{\{u,w\}: u,w \in \ch(h_i) \setminus \{v_i\} \} \cap A_{h_i}$ for all $i \in I^{(=)}$, and we define $\widetilde{\mathcal{A}}=(\widetilde{A}_h)_{h \in \mathcal{H}}$ to be the collection given by 
\[
    \widetilde{A}_{h_i} = \begin{cases}
        C_i  \cup \{S_i\} & \text{if } i \in I^{(=)}, \\
        C_i & \text{if } i \in I^{(<)}.
    \end{cases}
\]
By construction, this collection is valid and ZUTA-compliant. In particular, it is pairwise disjoint. Moreover, the sum of cardinalities is unchanged, that is, $\sum_{h \in \mathcal{H}}|\widetilde{A}_h| = \sum_{h \in \mathcal{H}}|A_h|$.
\end{proof}

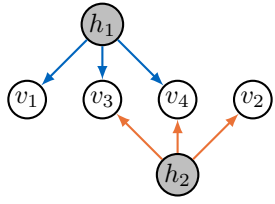
\begin{figure}[t]
\begin{center}
\begin{tikzpicture}[align=center]
    \node[fill=lightgray] (h1) at (-1.5,1) {$h_1$};
    \node[fill=lightgray] (h2) at (-0.5,-1) {$h_2$};

    \node[] (1) at (-2.5,0) {$v_1$};
    \node[] (2) at (-1.5,0) {$v_3$};
    \node[] (3) at (-0.5,0) {$v_4$};
    \node[] (4) at (0.5,0) {$v_2$};
    
    \draw[MyBlue] [-latex] (h1) edge (1);
    \draw[MyBlue] [-latex] (h1) edge (2);
    \draw[MyBlue] [-latex] (h1) edge (3);
    \draw[MyRed] [-latex] (h2) edge (2);
    \draw[MyRed] [-latex] (h2) edge (3);
    \draw[MyRed] [-latex] (h2) edge (4);

\end{tikzpicture}
\caption{A graph where each latent nodes has exactly one pure child.}
\label{fig:1-pure-children}
\end{center}
\end{figure}

\begin{example}
Consider the graph in Figure~\ref{fig:1-pure-children}, where both latent nodes have exactly one pure child.   We have
\begin{align*}
C(V,2)_{h_1} &= \{\{v_1,v_3\}, \{v_1,v_4\}, \{v_3,v_4\}\}, \\
C(V,2)_{h_2} &= \{\{v_2,v_3\}, \{v_2,v_4\}, \{v_3,v_4\}\},
\end{align*}
and both latent variables have three children.
It is easy to see that in any valid collection with maximal sum of cardinalities it must be that either $|A_{h_1}|=2$ or $|A_{h_2}|=2$. If both $A_{h_1}$ and $A_{h_2}$ have cardinality equal to $3=|\ch(h_i)|$, then we must have that $\{v_3,v_4\} \in A_{h_1} \cap A_{h_2}$, that is, the collection is not pairwise disjoint.
By Corollary~\ref{cor:1-pure-children}, the dimension is therefore given by 
\[
    \dim(F(G)) =  |V| + \left(\sum_{h \in \mathcal{H}} |\ch(h)|\right) - 1 = 4 + 6 - 1 = 9,
\]
which is one less than the expected dimension.
\end{example}

The next example considers a graph, where our upper and lower bound do not coincide, even after potential relabeling of the nodes. Note that the example  is already quite complex since it involves 5 latent variables and 9 observed variables. However, it is the smallest nontrivial example we could find where the lower bound is different from the upper bound.

\begin{example}
    Computations using our code on \href{https://mathrepo.mis.mpg.de/sparse-factor-analysis/}{MathRepo} show that the dimension of the model corresponding to the graph in Figure~\ref{fig:different-bounds} is $35$, which coincides with the expected dimension from counting parameters. It is easy to find a valid collection $\mathcal{A}=(A_h)_{h \in \mathcal{H}}$ that has sum of cardinalities $\sum_{h \in \mathcal{H}}|A_h|$ equal to the total number of children $\sum_{h \in \mathcal{H}}|\ch(h)|=26$. However, there are no relabelings of the latent and observed nodes such that ZUTA is satisfied and there is a \emph{ZUTA-compliant} valid collection that also has the sum of cardinalities equal to the total number of children. Hence, the lower bound from Theorem~\ref{thm:dim-lower-bound} is different than the upper bound from Theorem~\ref{thm:dimension-general}. For example, with the labeling as displayed in Figure~\ref{fig:different-bounds}, any valid, ZUTA-compliant collection has sum of cardinalities at most $23$. If we permute the labels of the nodes  $v_5$ and $v_6$ to the end, that is, the nodes $v_5$ and $v_6$ become $v_8$ and $v_9$, then it is possible to construct a ZUTA-compliant collection of cardinalities at most $24$, but this is still less than the total number of children.
\end{example}

By Theorem~\ref{thm:dim-lower-bound}, a model has expected dimension $|V|+|D|$ if is satisfies ZUTA and there is a ZUTA-compliant, valid collection $\mathcal{A}=(A_h)_{h \in \mathcal{H}}$ such that $A_h \subseteq C(V,2)_h$ has cardinality $|\! \ch(h)|$ for all $h \in \mathcal{H}$. Hence, a trivial necessary condition for expected dimension is that each latent node has at least three children. 
If a latent node $h \in \ch(h)$ has at most two children,  we have that $|C(V,2)_h| < |\ch(h)|$ and thus we must have that the cardinality of $A_h$ is strictly smaller than the number of children.
For a class of factor analysis graphs that satisfy stronger sparsity conditions than ZUTA, we obtain that the dimension is always equal to the expected dimension. Providing a lower bound that also holds for graphs violating ZUTA appears to be challenging, and we have not found a feasible approach that goes beyond case-by-case studies for each graph. 

\begin{corollary} \label{cor:2-children-dim}
    Let $G=(V \cup \mathcal{H}, D)$ be a factor analysis graph such that $|\ch(h)| \geq 3$ for all $h \in \mathcal{H}$. Moreover, assume that there exist relabelings of the latent and observed nodes such that $\mathcal{H}=\{h_1, \ldots, h_m\}$ and $V=\{v_1, \ldots, v_p\}$ and it holds that $v_{2i-1},v_{2i} \in \ch(h_i)$ and $v_{2i-1},v_{2i} \not\in \bigcup_{j > i} \ch(h_j)$ for all $i=1,\ldots,m$. Then, we have
    \[
        \dim(F(G)) = |V| + |D|.
    \]
\end{corollary}

\begin{proof} 
For every latent node $h_i$ there are at least three children. Two of them are given by $v_{2i-1}$ and $v_{2i}$ and we denote an arbitrary third child by $w_i$. Note that the children $v_{2i-1}$ and $v_{2i}$ are different for every $i \in [m]$, that is $\{v_{2i-1}, v_{2i}\} \cap \{v_{2j-1}, v_{2j}\} = \emptyset$ for $i \neq j$,  while the third child $w_i$ might also be a child of some other latent node $h_j$. In particular, it might be that $w_i=w_j$. We define a collection $\mathcal{A}=(A_h)_{h \in \mathcal{H}}$ by
\[
    A_{h_i} = \{\{v_{2i-1}, w\}: w \in \ch(h_i) \setminus \{v_{2i-1}\}\} \cup \{\{v_{2i}, w_i\}\} \subseteq C(V,2)_{h_i}.
\]
Clearly, the collection $\mathcal{A}$ is valid. It is also ZUTA-compliant if we relabel the nodes $v_{2i-1}$ to be $v_i$ for all $i \in [m]$. Since $|A_h|=|\ch(h)|$, the sum of cardinalities $\sum_{h \in \mathcal{H}} |A_h|$ is maximal and it is equal to   $\sum_{h \in \mathcal{H}} |\ch(h)| = |D|$.
\end{proof}

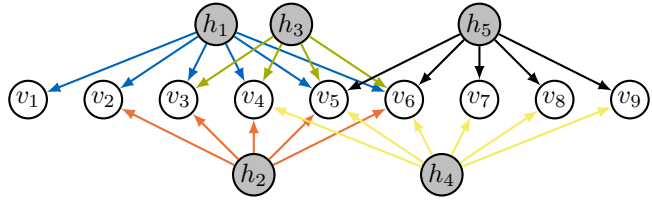
\begin{figure}[t]
\begin{center}
\begin{tikzpicture}[align=center]
    \node[fill=lightgray] (h1) at (0,1) {$h_1$};
    \node[fill=lightgray] (h2) at (0.5,-1) {$h_2$};
    \node[fill=lightgray] (h3) at (1,1) {$h_3$};
    \node[fill=lightgray] (h4) at (3,-1) {$h_4$};
    \node[fill=lightgray] (h5) at (3.5,1) {$h_5$};

    \node[] (1) at (-2.5,0) {$v_1$};
    \node[] (2) at (-1.5,0) {$v_2$};
    \node[] (3) at (-0.5,0) {$v_3$};
    \node[] (4) at (0.5,0) {$v_4$};
    \node[] (5) at (1.5,0) {$v_5$};
    \node[] (6) at (2.5,0) {$v_6$};
    \node[] (7) at (3.5,0) {$v_7$};
    \node[] (8) at (4.5,0) {$v_8$};
    \node[] (9) at (5.5,0) {$v_9$};
    
    \draw[MyBlue] [-latex] (h1) edge (1);
    \draw[MyBlue] [-latex] (h1) edge (2);
    \draw[MyBlue] [-latex] (h1) edge (3);
    \draw[MyBlue] [-latex] (h1) edge (4);
    \draw[MyBlue] [-latex] (h1) edge (5);
    \draw[MyBlue] [-latex] (h1) edge (6);
    \draw[MyRed] [-latex] (h2) edge (2);
    \draw[MyRed] [-latex] (h2) edge (3);
    \draw[MyRed] [-latex] (h2) edge (4);
    \draw[MyRed] [-latex] (h2) edge (5);
    \draw[MyRed] [-latex] (h2) edge (6);
    \draw[MyGreen] [-latex] (h3) edge (3);
    \draw[MyGreen] [-latex] (h3) edge (4);
    \draw[MyGreen] [-latex] (h3) edge (5);
    \draw[MyGreen] [-latex] (h3) edge (6);
    \draw[MyYellow] [-latex] (h4) edge (4);
    \draw[MyYellow] [-latex] (h4) edge (5);
    \draw[MyYellow] [-latex] (h4) edge (6);
    \draw[MyYellow] [-latex] (h4) edge (7);
    \draw[MyYellow] [-latex] (h4) edge (8);
    \draw[MyYellow] [-latex] (h4) edge (9);
    \draw [-latex] (h5) edge (5);
    \draw [-latex] (h5) edge (6);
    \draw [-latex] (h5) edge (7);
    \draw [-latex] (h5) edge (8);
    \draw [-latex] (h5) edge (9);
\end{tikzpicture}
\caption{A graph where lower and  upper bound on the dimension do not coincide.}
\label{fig:different-bounds}
\end{center}
\end{figure}

Note that none of the graphs in Figure~\ref{fig:dimension} satisfies the condition in Corollary~\ref{cor:2-children-dim}. But the graph in Figure~\ref{fig:graph_5_2_4} satisfies the condition if we swap, for example, the label of nodes $v_3$ and $v_6$.

\begin{remark}
    One might be tempted to compare the condition in Corollary \ref{cor:2-children-dim} to the often employed ``$2$-pure-children'' condition, recall Remark \ref{rem:pure-children}. The $2$-pure children condition requires that each latent node $h \in \mathcal{H}$ has at least $2$ children that have no other parents than $h$. However, the condition in Corollary \ref{cor:2-children-dim} is equivalent to the requirement that, for each $i \in [m]$, the latent node $h_i$ has two pure children  in the subgraph obtained by deleting all latent nodes with an index smaller than $i$. Hence, the children of $h_i$ in the original graph are generally allowed to have more than one parent. For example, after relabeling, nodes $v_5$ and $v_6$ need to be children of $h_3$, but $h_1$ and $h_2$ could also be parents of $v_5$ or $v_6$. The condition only requires that all $h_j$ with $j>3$ are not parents of $v_5$ and $v_6$. Said differently, only $v_1$ and $v_2$ are \emph{pure} children of $h_1$ in the classical sense. Conversely, if there is no latent node that has two pure children in the classical sense, then the factor analysis graph does not satisfy the condition in Corollary \ref{cor:2-children-dim}.
\end{remark}

\section{Algebraic Invariants of Sparse Two-Factor Models} \label{sec:invariants}
We are interested in polynomial invariants that hold on a covariance matrix $\Sigma \in F(G)$, where $G$ is a sparse factor analysis graph. For any subset $F \subseteq \PD(p)$, the \emph{ideal of invariants} is defined as 
\[
I(F) = \{f \in \mathbb{R}[\sigma_{ij}, \ i \leq j]: f(\Sigma)=0 \text{ for all } \Sigma \in F\}.
\]
Our object of interest is the ideal of invariants of sparse factor analysis models. Since, for a symmetric positive definite matrix $\Sigma \in \mathbb{R}^{p \times p}$, membership in $F(G)$ only depends on the off-diagonal entries of $\Sigma$, we can regard the ideal of invariants of $F(G)$, i.e., $I(G):=I(F(G))$ as an ideal in the subring $\mathbb{R}[\sigma_{ij},\ i < j]$. It is our goal to find a finite set of polynomials that generate $I(G)$. If a factor analysis graph has an edge to every observed node, the model is equivalent to a \emph{full} factor analysis model. In the case of one or two latent nodes, the ideal of invariants is then completely understood, see \cite[Theorem 16]{AlgebraicFactorAnalysis} and \cite{SecondHypersimplex}. However, finding a minimal set of generators or a Gröbner basis for the full factor analysis model with three latent nodes is, to the best of our knowledge, still an open problem. \\

First, we consider the special case where the children sets $\ch(h_i)$ of a factor analysis graph $G=(V \cup \mathcal{H}, D)$ only intersect in at most one node.  Our next proposition reveals that in this case the ideal of invariants is a sum of ideals obtained from induced subgraphs corresponding to full \emph{one}-factor analysis models. 
Note that the  ideal of the full one-factor analysis model is toric, that is, it is prime and binomial; see Theorem \ref{thm: Gröbner basis of the full one-factor model}. In particular, it is the toric ideal of the edge subring of the complete graph on the observable node set $V$.

\begin{proposition}
\label{prop:small_intersections}
Let $G=(V \cup \mathcal{H}, D)$ be a factor analysis graph such that for any disjoint pair $(h_i, h_j) \in \mathcal{H} \times \mathcal{H} \backslash \{(h_1, h_m)\}$  of latent nodes,
we have that $\lvert \ch(h_i) \cap \ch(h_j) \rvert \leq 1$ and $\lvert \ch(h_1) \cap \ch(h_m) \rvert =0$. Let $G_i$ be the induced subgraph $G[\{h_i\} \cup \ch(h_i)] \subseteq G$ on the  vertex set $\{h_i\} \cup \ch(h_i)$, for $i \in [m]$.
Then we obtain that 
$
I(G) = I(G_1) + \dots + I(G_m) + \langle \sigma_{ij}: \pa(i) \cap \pa(j) = \emptyset \rangle 
$ and it is toric.

\end{proposition}

\begin{proof}
The off-diagonal entries of the parametrization $\tau_G$ given in \eqref{eq: parametrization} are monomial and $\sigma_{ij} = 0$, if  $\pa(i) \cap \pa(j) = \emptyset$ for $i \neq j$. Moreover, the ideals of invariants of submatrices of $\Sigma$ that correspond to the covariance models of induced subgraphs $G[\{h_i\} \cup \ch(h_i)] \subseteq G$ are toric \cite[Theorem 16]{AlgebraicFactorAnalysis}.  Finally, since G is a polytree (a directed acyclic graph whose underlying undirected graph is a tree) or a union of polytrees, the ideal is toric. 
This fact may be proven in the same way as \cite[Proposition 5.4]{ThirdOrderMoments}
but intersecting only with variables corresponding to the covariances.
\end{proof}

If $\lvert \ch(h_i) \cap \ch(h_j) \rvert \geq 2$ for some disjoint pairs $(h_i, h_j) \in \mathcal{H} \times \mathcal{H}$, the ideal $I(G)$ may not be toric. In Section~\ref{subsec:variety} and \ref{subsubsection: sparse two}, we explore this case for two-factor models.

\subsection{Variety} \label{subsec:variety} In this section, we focus on $|\mathcal{H}|=2$, that is, factor analysis graphs $G$ with two latent nodes. We propose generators of an ideal such that the variety corresponding to the ideal is the smallest variety that contains the model. Recall that for any ideal $I \subseteq \mathbb{R}[x_1, \ldots, x_n]$, the corresponding variety (over the complex numbers) is defined as $\mathbb{V}(I) = \{x \in \mathbb{C}^n: f(x)=0 \text{ for all } f \in I\}$.
We say that a set of latent nodes $H \subseteq \mathcal{H}$ \emph{separates} two sets of observed nodes $A,B \subseteq V$ if $\pa(A) \cap \pa(B) \subseteq H$, where we use the notation $\pa(A)=\bigcup_{a \in A}\pa(a)$. Importantly, the set $H$ might also be empty, that is, $A$ and $B$ are separated given the empty set if they do not have a joint parent. We denote the submatrix of $\Sigma$ given by the rows $A$ and columns $B$ by $\Sigma_{A,B}$.

\begin{definition}
Let $G=(V \cup \mathcal{H}, D)$ be a factor analysis graph. The ideal $M_{\leq 1}(G)  \subseteq \mathbb{R}[\sigma_{ij},\ i \leq j]$ is generated by all minors $\det(\Sigma_{A,B})$, where $A,B \subseteq V$ are two sets of observed nodes with cardinality $|A| = |B| \leq 2$ and there is $H \subseteq \mathcal{H}$ with $|H| < |A|$ such that $H$ separates $A$ and $B$.
\end{definition}

In other words, the ideal $M_{\leq 1}(G)$ is generated by minors $\det(\Sigma_{A,B})$ such that $A$ and $B$ are separated by at most one latent factor.
Note that $A$ and $B$ are not necessarily disjoint. 

\begin{example}
Consider the graph from Figure~\ref{fig:graph_5_2_4}. Let $A=\{1\}$, $B=\{7\}$ and $H=\emptyset$. Since the nodes $1$ and $7$ are separated by the empty set, i.e., $\pa(1) \cap \pa(7) = \emptyset$, we have that the monomial $\sigma_{17}$ is in the generating set of the ideal  $M_{\leq 1}(G)$. The sets $A=\{1,2\}$ and $B=\{4,5\}$ are separated by $H=\{h_1\}$ since $\pa(A) \cap \pa(B) = \{h_1\}.$ Thus, the minor $\sigma_{14} \sigma_{25} - \sigma_{24} \sigma_{15}$ is a generator of $M_{\leq 1}(G)$. On the other hand, the sets $A=\{1,4\}$ and $B=\{2,5\}$ can only be separated by $H=\{h_1, h_2\}$, that is, we need at least two latent factors for separation. This yields that the minor $\sigma_{12} \sigma_{45} - \sigma_{24} \sigma_{15}$ is \emph{not} in the generating set of $M_{\leq 1}(G)$.
\end{example}

Let $M_{p,m} \subseteq \mathbb{R}[\sigma_{ij},\ i \leq j]$ be the ideal that is generated by all $(m+1) \times (m+1)$-minors of a symmetric matrix $\Sigma \in \mathbb{R}^{p \times p}$. 

\begin{theorem} \label{thm:variety-of-model}
Let $G=(V \cup \mathcal{H}, D)$ be a factor analysis graph with $|\mathcal{H}|=2$ latent factors. Then 
\[
    \mathbb{V}(I(G)) = \mathbb{V}((M_{p,2} +  M_{\leq 1}(G)) \cap \mathbb{R}[\sigma_{ij},\ i < j]),
\]
where the varieties $\mathbb{V}(\cdot)$  are understood over the field $\mathbb{C}$ of complex numbers.
\end{theorem}
\begin{proof}

Let $\mathcal{L} = \{\Lambda \Lambda^{\top} \in \mathbb{R}^{p \times p}: \Lambda \in \mathbb{R}^D\}$. We first prove that $\mathbb{V}(M_{p,2} +  M_{\leq 1}(G))$ is equal to the Zariski closure $\overline{\mathcal{L}}$ in $\mathbb{C}^{p \times p}$. For the inclusion $\overline{\mathcal{L}} \subseteq \mathbb{V}(M_{p,2} + M_{\leq 1}(G))$, consider a matrix $\Sigma \in \mathcal{L}$. Then $\Sigma$ is symmetric and $\Sigma \in \mathbb{V}(M_{p,2})$. Moreover one can check that $\Sigma \in \mathbb{V}(M_{\leq 1}(G))$ by applying trek separation \cite{sullivant2010trek}. Thus, $\mathcal{L} \subseteq \mathbb{V}(M_{p,2} +  M_{\leq 1}(G))$ and since the variety $\mathbb{V}(M_{p,2} +  M_{\leq 1}(G))$ is Zariski closed, we obtain that $\overline{\mathcal{L}} \subseteq \mathbb{V}(M_{p,2} + M_{\leq 1}(G))$. 

For the other direction, assume that $\Sigma \in \mathbb{V}(M_{p,2} +  M_{\leq 1}(G))$. We explicitly construct a matrix $\Lambda \in \mathbb{C}^D$ such that $\Sigma=\Lambda\Lambda^{\top}$. 
We assume that there is no node $v \in V$ such that $\pa(v) = \emptyset$, since this case is trivial where the row of $\Lambda$ that is indexed by $v$ is zero. Let $\mathcal{H}=\{h_1,h_2\}$ and $V=V_1 \dot\cup V_2 \dot\cup V_3$ be a partition of the observed nodes $V$ into three subsets such that 
$V_1 = \ch(h_1) \setminus \ch(h_2)$, $V_3 = \ch(h_2) \setminus \ch(h_1)$ and $V_2= V \setminus (V_1 \cup V_3)$. 
Without loss of generality we assume that there exists a node $v \in V_1$ such that $\sigma_{vv}\neq0$. 
We fix this node $v$ and define the matrices
    \[
    \Lambda_1= \begin{blockarray}{cc}
    \begin{block}{(c|c)}
    \mathbf{x} & 0 \\
    \BAhline
    \mathbf{y}  & I_{|V_2|+|V_3|} \\
    \end{block}
    \end{blockarray}
    \in \mathbb{C}^{p \times (1+|V_2|+|V_3|)}
    \quad  \text{and} \quad
    \Sigma_1= \begin{blockarray}{cc}
    \begin{block}{(c|c)}
    1 & 0 \\
    \BAhline
   0 & A \\
    \end{block}
    \end{blockarray}
    \in \mathbb{C}^{(1+|V_2|+|V_3|)\times (1+|V_2|+|V_3|)}.
    \]
The vector $\mathbf{x}  \in \mathbb{C}^{V_1}$ is defined by $\mathbf{x}_w=\sigma_{vw}/\sqrt{\sigma_{vv}}$ for $w \neq v$ and by $\mathbf{x}_v=\sqrt{\sigma_{vv}}$ for the node $v$. The vector $\mathbf{y}  \in \mathbb{C}^{V_2 \cup V_3}$ is defined by $\mathbf{y} _w=\sigma_{vw}/\sqrt{\sigma_{vv}}$ if $w \in V_2$ and by $\mathbf{y} _w=0$ else. Finally, the symmetric matrix $A=(a_{uw}) \in \mathbb{C}^{V_2 \cup V_3, V_2 \cup V_3}$ is defined as $A=\Sigma_{V_2 \cup V_3, V_2 \cup V_3} - \mathbf{y}  \mathbf{y} ^{\top}$. 

We prove next that $\Lambda_1\Sigma_1\Lambda_1^{\top}=\Sigma$. The essential step is to exploit that $\Sigma \in  M_{\leq 1}(G)$, which implies the three properties:
\begin{equation} \label{eq:sigma-props}
    \Sigma_{{V_1},{V_3}}=0, \quad \rk(\Sigma_{V_1, V_1 \cup V_2})\leq1, \quad \text{and \quad} \rk(\Sigma_{V_2 \cup V_3, V_3})\leq1.
\end{equation}
It holds that $[\Lambda_1\Sigma_1\Lambda_1^{\top}]_{{V_1},{V_3}}=\Sigma_{{V_1},{V_3}}$, since $[\Lambda_1\Sigma_1\Lambda_1^{\top}]_{{V_1},{V_3}}= \mathbf{x}  \mathbf{y}_{V_3}^{\top}=0$.  Next, we show that $[\Lambda_1\Sigma_1\Lambda_1^{\top}]_{V_1,V_1 \cup V_2} = \Sigma_{V_1,V_1 \cup V_2}$. For any node $w \in V_1 \setminus \{v\}$, we have that $[\Lambda_1\Sigma_1\Lambda_1^{\top}]_{vw}=\mathbf{x}_w \mathbf{x}_v = (\sigma_{vw}/\sqrt{\sigma_{vv}}) \sqrt{\sigma_{vv}} = \sigma_{vw}$. Let $k,w \in V_1\setminus \{v\}$. By \eqref{eq:sigma-props}, it holds that $\det(\Sigma_{\{v,k\},\{v,w\}})=0$, i.e.\ $\sigma_{vv}\sigma_{kw}=\sigma_{vk}\sigma_{vw}$. Hence, 
\[
[\Lambda_1\Sigma_1\Lambda_1^{\top}]_{kw}=\mathbf{x}_k \mathbf{x}_w = \frac{\sigma_{vk} \sigma_{vw}}{\sigma_{vv}} = \sigma_{kw}.
\]
If $w$ is an element of $V_2$ instead of $V_1$, the conclusion follows similarly by replacing $\mathbf{x}_w$ with $\mathbf{y}_w$. We finally observe that  the equality $[\Lambda_1\Sigma_1\Lambda_1^{\top}]_{V_2 \cup V_3, V_2 \cup V_3} = \Sigma_{V_2 \cup V_3, V_2 \cup V_3}$ follows directly from the definitions of $\Lambda_1$ and $\Sigma_1$.

Now, we return to proving that $\Sigma = \Lambda_1 \Sigma_1 \Lambda_1^{\top}\in \overline{\mathcal{L}}$. The matrix $\Lambda_1$ has full rank equal to $1+|V_2|+|V_3|$ and the rank of $\Sigma$ is at most $2$. By Sylvester's rank inequality, this implies $\rk(\Sigma_1)\leq2$. In particular, we have that $\rk(A)\leq1$. 
Without loss of generality we may assume that there is a node $u \in V_2 \cup V_3$ such that that $a_{uu}\neq 0$. We fix this node $u$ and define the matrix 
\[
    \Lambda_2 = \begin{blockarray}{cc}
    \begin{block}{(c|c)}
    1 & 0 \\
    \BAhline
    0  & \mathbf{z} \\
    \end{block}
    \end{blockarray},
\]
where $\mathbf{z} \in \mathbb{C}^{V_2 \cup V_3}$ is defined by $\mathbf{z}_w = a_{wu}/\sqrt{a_{uu}}$ for $w \neq u$ and by $\mathbf{z}_u = \sqrt{a_{uu}}$ for the node $u$. Using similar arguments as above, it is easy to see that $\Sigma_1 = \Lambda_2 \Lambda_2^{\top}$. Finally, define $\Lambda=\Lambda_1 \Lambda_2$ and observe that
\[
    \Lambda = \begin{blockarray}{cc}
    \begin{block}{(c|c)}
    \mathbf{x} & 0 \\
    \BAhline
    \mathbf{y}_{V_2}  & \mathbf{z}_{V_2} \\
    \BAhline
    0  & \mathbf{z}_{V_3} \\
    \end{block}
    \end{blockarray} \in \mathbb{C}^D.
\]
This shows that $\Sigma \in \overline{\mathcal{L}}$ since
$
    \Lambda \Lambda^{\top} = \Lambda_1 \Lambda_2 \Lambda_2^{\top} \Lambda_1^{\top} = \Lambda_1 \Sigma_1 \Lambda_1^{\top} = \Sigma.
$

We now prove the statement of the theorem. Consider the projection $\pi$ of the space of symmetric $p \times p$ matrices onto the space of the off-diagonal entries. We have that
\[
    I(G)=I(F(G))=I(\pi(F(G)))=I(\overline{\pi(F(G))})=I(\overline{\pi(\mathcal{L})}),
\]
where the second equality follows from the fact that membership in $I(F(G))$ only depends on the off-diagonal entries, also see \cite{brouwer2011}. Since the Zariski closure of the projection of an arbitrary set is equal to the Zariski closure of the projection of the Zariski closure of the set, it follows that $\mathbb{V}(I(G)) = \overline{\pi{(\overline{\mathcal{L}})}}$. 
Consequently, we have that $\pi{(\overline{\mathcal{L}})} = \pi(\mathbb{V}(M_{p,2} +  M_{\leq 1}(G)))$ and by \cite[§4.4, Theorem 4]{Cox97}, the Zariski closure of the projection $\pi(\mathbb{V}(M_{p,2} +  M_{\leq 1}(G)))$ is $\mathbb{V}((M_{p,2} +  M_{\leq 1}(G)) \cap \mathbb{R}[\sigma_{ij},\ i < j])$.
\end{proof}

It was shown in \cite{brouwer2011} that the ideal $M_{p,2} \cap \mathbb{R}[\sigma_{ij},\ i < j]$ is generated by two types of generators: off-diagonal $3\times3$-minors and certain polynomials of degree $5$ known as {\emph{pentads}} \cite{kelley1935essential}. Thus, it is natural to conjecture that the ideal $(M_{p,2} +  M_{\leq 1}(G)) \cap \mathbb{R}[\sigma_{ij},\ i < j]$ is generated by off-diagonal $3\times3$-minors, pentads, and the off-diagonal $1\times1$ and $2\times2$-minors in $M_{\leq 1}(G)$; see Conjecture \ref{conj:generators}.

Theorem 3.4 implies that the ideal $(M_{p,2} +  M_{\leq 1}(G)) \cap \mathbb{R}[\sigma_{ij},\ i < j]$ is included in the ideal of invariants $I(G)$ we are interested in. In the next section, we combinatorially find a Gr\"obner basis of $I(G)$ for the special case where the two-factor analysis model has overlap two, that is, there are at most two observed nodes that have two latent parents and all other observed nodes have at most one latent parent. We obtain as Corollary \ref{cor:proposed-equals-vanishing} that we have indeed $I(G)=(M_{p,2} +  M_{\leq 1}(G)) \cap \mathbb{R}[\sigma_{ij},\ i < j]$, and that this ideal is generated by the concerned polynomials that vanish on the model.

Our readers are encouraged to use our code on \href{https://mathrepo.mis.mpg.de/sparse-factor-analysis/}{MathRepo} to experiment with Gröbner basis computations of $I(G)$. Gr\"obner bases for the \emph{full} factor analysis model with one and two latent nodes are given in \cite{de1995grobner} and \cite{SecondHypersimplex}. \looseness=-1

\begin{figure}[t]
\centering
    \begin{tikzpicture}
        \node[fill=lightgray] (h1) at (0,1) {$h_1$};

        \node[] (1) at (-2,0) {$v_1$};
        \node[] (2) at (-1,0) {$v_2$};
        \node[] (3) at (0,0) {$v_3$};
        \node[] (4) at (1,0) {$v_4$};
        \node[] (5) at (2,0) {$v_5$};
        \node[] (6) at (3,0) {$v_6$};
        \node[] (7) at (4,0) {$v_7$};
        
        \draw[MyBlue] [-latex] (h1) edge (1);
        \draw[MyBlue] [-latex] (h1) edge (2);
        \draw[MyBlue] [-latex] (h1) edge (3);
        \draw[MyBlue] [-latex] (h1) edge (4);
        \draw[MyBlue] [-latex] (h1) edge (5);
    \end{tikzpicture}\quad\quad
    \begin{tikzpicture}
        
        \node[fill=lightgray] (h2) at (2.5,-1) {$h_2$};
        
        \node[] (1) at (-2,0) {$v_1$};
        \node[] (2) at (-1,0) {$v_2$};
        \node[] (3) at (0,0) {$v_3$};
        \node[] (4) at (1,0) {$v_4$};
        \node[] (5) at (2,0) {$v_5$};
        \node[] (6) at (3,0) {$v_6$};
        \node[] (7) at (4,0) {$v_7$};

        \draw[MyRed] [-latex] (h2) edge (4);
        \draw[MyRed] [-latex] (h2) edge (5);
        \draw[MyRed] [-latex] (h2) edge (6);
        \draw[MyRed] [-latex] (h2) edge (7);
    \end{tikzpicture}    
  \caption{Two sparse one-factor analysis graphs where $\{4,5\}$ is the only pair of vertices such that its joint parents are the set of latent nodes $\mathcal{H} = \{h_1, h_2\}$. Identifying these two graphs via the observed nodes $V = \{v_1, \ldots, v_7\}$ yields the two-factor sparse analysis graph from Figure~\ref{fig:graph_5_2_4}.} 
\label{fig: separating of two one-factors}
\end{figure}
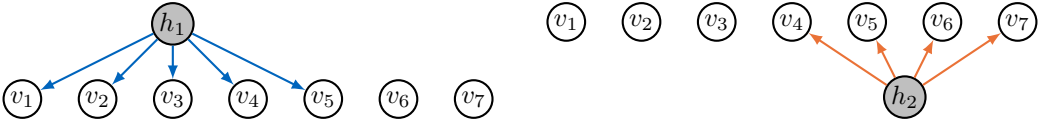

\subsection{Gröbner Basis} \label{subsec:groebner-basis}
In this section, we start by introducing a reduced Gröbner basis for sparse one-factor analysis models. We then use this to construct a Gröbner basis for the ideal of invariants for certain two-factor analysis models. This approach involves considering the join of sparse one-factor models with the "2-delightful" approach from \cite{combinatorialsecant, SecondHypersimplex}.

\subsubsection{Sparse one-factor models} We recall the definition for the toric ideal of the edge subring of a graph. Let $G$ be a simple undirected graph on $p$ vertices. We define the {\emph{edge ring}} associated to $G$ as
\[
\text{Edr}(G) := \mathbb{C}[t_it_j \mid \{i, j\} \in E(G)]
\] and consider the surjective ring homomorphism:
\begin{align*}
\begin{split}
\Phi_G \colon \mathbb{C}[\sigma_{ij} | 1 \leq i < j \leq p] &\longrightarrow \text{Edr}(G)\\
\sigma_{ij} &\mapsto \begin{cases}
    t_i t_j & \text{if } ij \in E(G), \\
    0  & \text{otherwise}.
\end{cases}
\end{split}
\end{align*}
The kernel of $\Phi_G$ is called the \emph{toric ideal (of the edge subring) of} $G$. If $G = K_p$ is the complete graph on $p$ vertices, then $\ker(\Phi_G)$ is the ideal $I_{p,1}$ of the full one-factor analysis model. This immediately follows by the parametrization (\ref{eq: parametrization}). This ideal is also called the \emph{ideal for the second hypersimplex}. 

For a Gröbner basis of $I_{p,1}$, we consider a circular embedding of the complete graph $K_p$ where vertices are presented as the $p$-th roots of unity in the complex plane. The edges of $K_p$ belong to the $\lfloor \frac{p}{2} \rfloor$ orbits under the action of dihedral group $D_p$ on the roots of unity. The $k$-th class of edges is the set of edges that are equivalent to the edge $1k$ for $k \in \{2, \ldots, \lceil \frac{p}{2}\rceil\}$. In other words, the edges that are closer to the boundary of the circular embedding correspond to larger variables in the block ordering.  
\begin{definition}
    A circular term order is any block term order such that $\sigma_{i_1 j_1} \succ\sigma_{i_2 j_2}$ whenever the edge $i_1 j_1$ is in a smaller class than the edge $i_2 j_2$.
\end{definition}
 
The Gröbner basis for the ideal for the second hypersimplex, or equivalently for $I_{p,1}$, with respect to any circular order is studied by De Loera, Sturmfels, and Thomas. 

\begin{theorem}[{\cite[Theorem 2.1]{de1995grobner}}]\label{thm: Gröbner basis of the full one-factor model}
The set of square-free quadratic binomials
\begin{equation}\label{eq: tetrads}
 \{\underline{\sigma_{ij} \sigma_{kl}} - \sigma_{ik} \sigma_{jl}, \underline{\sigma_{il}\sigma_{jk}} - \sigma_{ik} \sigma_{jl} \ | \ 1 \leq i < j < k < l \leq n \}
\end{equation} 
is a reduced Gröbner basis for the one-factor analysis model $I_{p,1}$ 
with respect to any circular term order. 
\end{theorem}
These square-free quadratic binomials are known as \emph{tetrads} in the statistics literature. We first adapt this result to sparse one-factor analysis models. Let us consider a sparse one-factor analysis graph where $A \subseteq V$ is the set of children of the latent node and $B = [p]\backslash A$, i.e., the set of isolated vertices. We denote the ideal of invariants of a sparse one-factor analysis model as $I_{A, B ,1}$. The ideal $I_{A, B ,1}$ is the toric ideal of the complete graph $K_{|A|}$ on the vertex set $A$ with the set $B$ of isolated vertices. Thus, one needs to add $|A||B| + \binom{|B|}{2}$ degree-one monomials to the set in Theorem~\ref{thm: Gröbner basis of the full one-factor model} to form a reduced Gröbner basis for $I_{A,B,1}$. To simplify the next statement, we relabel the vertices of $A$ as $1,\ldots,|A|$ and the vertices of $B$ as $|A| +1, \ldots, p$.

\begin{proposition}\label{prop: Gröbner basis for sparse one-factor}
The set of degree-one monomials and tetrads
$$ \{\sigma_{ij} \ | \  i \in B \text{ or } j \in B \} \cup \{\underline{\sigma_{ij} \sigma_{kl}} - \sigma_{ik} \sigma_{jl}, \underline{\sigma_{il}\sigma_{jk}} - \sigma_{ik} \sigma_{jl} \ | \ 1 \leq i < j < k < l \leq |A|\}.$$
is a reduced Gröbner basis for the sparse one-factor analysis model $I_{A,B,1}$ with respect to any circular term order.
\end{proposition}

\begin{figure}[ht]

\begin{tikzpicture}[scale=2]
  \graph [nodes={circle, draw, inner sep=1pt, minimum size=1em}, clockwise, radius=1.3cm, n=5] {
    1/"1",
    2/"2",
    3/"3",
    4/"4",
    5/"5"
  };
  \node[circle, draw,inner sep=1pt, minimum size=1em] (16) at (1.5,0.4) {6};

  \node[circle, draw,inner sep=1pt, minimum size=1em] (67) at (1.5,0) {7};

  \draw[very thick, MyBlue] (1) -- (2);
  \draw (1) -- (3);
  \draw[very thick, MyRed]  (1) -- (4);
  \draw[very thick, MyGreen] (1) -- (5);
  \draw (2) -- (3);
  \draw[very thick, MyGreen] (2) -- (4);
  \draw[very thick, MyRed] (2) -- (5);
  \draw (3) -- (4);  
  \draw (3) -- (5);
  \draw[very thick, MyBlue] (4) -- (5);
  \end{tikzpicture}
\quad\quad \vline \quad \quad \begin{tikzpicture}[scale=2]
  \graph [nodes={circle, draw, inner sep=1pt, minimum size=1em}, clockwise, radius=1.3cm, n=10] {
    34/"34",
    35/"35",
    45/"45",
    12/"12",
    13/"13",
    14/"14",
    15/"15",
    23/"23",
    24/"24",
    25/"25"
  };
  \node[circle, draw,inner sep=1pt, minimum size=1em] (16) at (1.5,0.55) {16}; 

\node[draw=none,inner sep=1pt, minimum size=1em,scale=2] (16) at (1.5,0.1) {$\vdots$}; 

  \node[circle, draw,inner sep=1pt, minimum size=1em] (67) at (1.5,-0.5) {67}; 

  \draw (12) -- (34);
  \draw (12) -- (35);
  \draw (12) -- (45);
  \draw (13) -- (45);
  \draw (14) -- (23);
  \draw (12) -- (34);
  \draw (15) -- (23);
  \draw (15) -- (24);  
  \draw (15) -- (34);
  \draw (23) -- (45);
  \draw (25) -- (34);
  \end{tikzpicture}

\caption{The ideal of the sparse one-factor analysis model associated to the left graph in Figure~\ref{fig: separating of two one-factors} is the toric ideal of the graph on the left-hand side consisting of the complete graph $K_{|A|}$ and $|B|=2$ isolated vertices. The right-hand side pictures the corresponding initial ideal graph $G_{A,B}$ with 11 isolated vertices, where $A = \{1,2,3,4,5\}$ and $B = \{ 6,7\}$.}
\label{fig: complete graph with isolated vertices initial ideal graph}
\end{figure}
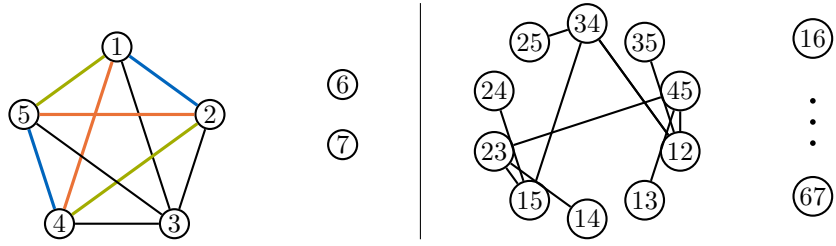
\begin{example}\label{ex: toric edge ideal of the complete graph with isolated vertices}
For the left model in Figure~\ref{fig: separating of two one-factors} with $A=\{1,2,3,4,5\}$ and $B= \{6,7\}$, the ideal $I_{A,B,1}$ is the toric ideal of the graph depicted on the left of Figure~\ref{fig: complete graph with isolated vertices initial ideal graph}. There are $11$ degree-one monomials $$\{\sigma_{16}, \sigma_{17}, \sigma_{26}, \sigma_{27}, \sigma_{36}, \sigma_{37}, \sigma_{46}, \sigma_{47}, \sigma_{56}, \sigma_{57}, \sigma_{67}\}.$$ 
and $2\binom{|A|}{4} = 10$ tetrads which form a reduced Gröbner basis for $I_{A,B,1}$ with respect to any circular term order. An easy way to construct these tetrads is by looking at the subgraph induced by the vertex set $\{i,j,k,l\}$. We say that a pair of edges $ij$, $kl$ cross if the line segments in the
circular drawing intersect (also at the endpoints) in the circular embedding. The noncrossing pairs (blue and green) of edges correspond to the leading terms of the tetrad generators and the crossing edges (orange) correspond to the remaining monomial of the tetrad. For $\{1,2,4,5\}$, we obtain the tetrads $\underline{\color{MyBlue}\sigma_{12} \sigma_{45}} - {\color{MyRed}\sigma_{14} \sigma_{25}}$ and $\underline{\color{MyGreen}\sigma_{24} \sigma_{15}} - {\color{MyRed}\sigma_{14} \sigma_{25}}$.
\end{example}

\subsubsection{Sparse two-factor models}\label{subsubsection: sparse two} In the case of full-factor analysis, the ideal of invariants is the $m$-th secant ideal of $I_{p,1}$. However, in the case of sparse-factor analysis models, we need to consider the \emph{join} of the ideals of sparse one-factor analysis models; see, e.g.  \cite{combinatorialsecant, miller2005combinatorial} for a rigorous definition of secants and joins. We consider a sparse $k$-factor analysis graph as $k$ sparse one-factor analysis graphs that are identified at their observed nodes $V$. Alternatively, these sparse one-factor analysis graphs can be seen as the induced subgraphs $G[\{h_i\} \cup V]$, where $h_i \in \mathcal{H}$ for $i \in [m]$; see Figure~\ref{fig: separating of two one-factors}. \\

To construct a Gröbner basis for a sparse two-factor analysis model with respect to any circular term order, we follow an analogous ``2-delightful" strategy which was used for the full two-factor analysis models in \cite{SecondHypersimplex}. For this, we first need to describe the join of the initial ideals of sparse one-factor analysis models. The initial ideal $\initial_{\prec}(I_{\{A,B\},1})$ is generated by all noncrossing pairs in the circular embedding of the complete graph $K_{|A|}$ on the vertex set $A$ and the degree-one monomials from Proposition~\ref{prop: Gröbner basis for sparse one-factor}. Thus it is the monomial edge ideal of a certain graph with isolated vertices.

\begin{definition}\label{def: new graph}
We define the simple graph called the {\emph{initial ideal graph}} $G_{A,B}$ whose vertices are labeled as $ij$ where $\{i,j\} \in A \sqcup B = [p]$ and $\{ij, kl\} \in E(G_{A,B})$ whenever $(ij, kl)$ is a noncrossing pair in the circular embedding of the complete graph $K_{|A|}$.
\end{definition}

To avoid confusion, the edges of the initial ideal graph are denoted by $\{ \ , \ \}$, different from the edges of the complete graph $K_{|\ch(j_i)|}$ with isolated vertices $\ch(h_i)^{\complement}$. The following definition describes how identifying two sparse one-factor analysis graphs via the observed nodes corresponds to identifying the two associated initial ideal graphs. We focus on the case where the ``overlap'' is two, that is, $|A_1 \cap A_2| = 2$.

\begin{definition}
   Let $G_{A_1, B_1}$ and $G_{A_2, B_2}$ be two initial ideal graphs with $A_1 \sqcup B_1 = A_2 \sqcup B_2 = V$ and $A_1 \cap A_2 = \{j_1, j_2\}$. We construct a \emph{glued hypergraph identified via $V$} denoted by $G_{A_1, B_1} \times_{V} G_{A_2,B_2}$ as follows: 
   \begin{itemize}
       \item The vertex set is $V(G_{A_1,B_1}) \cup V(G_{A_2,B_2}) \backslash \{\text{isolated vertices of } V(G_{A_1,\emptyset}) \cup V(G_{A_2,\emptyset})  \}$.
       \item The hyperedges of size 2 are 
        all those of $G_{A_1, B_1}$ and $G_{A_2,B_2}$ which do not contain vertex $j_1 j_2$.
       \item The hyperedges of size 3 are those $\{i, j_1 j_2, k\}$, where $i \in V(G_{A_1, \emptyset})$ and and $ k \in V(G_{A_2,\emptyset})$.
   \end{itemize}
\end{definition}
Note that the glued hypergraph $G_{A_1, B_1} \times_{V} G_{A_2,B_2}$ has $|A_1 \backslash \{j_1, j_2\}||A_2 \backslash \{j_1, j_2\}|$ isolated vertices $xy$ where $x \in A_1 \backslash \{j_1, j_2\}$ and $y \in A_2 \backslash \{j_1, j_2\}$. In particular, these correspond to the degree-one monomials $M_{\leq 1}(G)$ from Theorem~\ref{thm:variety-of-model}.

\begin{example}
The initial ideal graph for the sparse one-factor analysis model with $A_1=\{1,2,3,4,5\}$ and $B_1 =\{6,7\}$ is depicted on the right of Figure~\ref{fig: complete graph with isolated vertices initial ideal graph}. Identifying two sparse one-factor analysis graphs from Figure~\ref{fig: separating of two one-factors} gives rise to the sparse two-factor analysis graph from Figure~\ref{fig:graph_5_2_4}. Here we have $A_2 = \{4,5,6,7\}$ and $B_2 = \{1,2,3\}$. This corresponds to identifying the initial ideal graphs $G_{A_1,B_1}$ and $G_{A_2,B_2}$ via the vertex $45$ as in Figure~\ref{fig: identifying initial ideal graphs}. ${\color{MyGreen}\{12,45,67\}}, {\color{MyRed}\{23,45,67\}}, {\color{MyBlue}{\{13,45,67\}}}$ are the hyperedges of size 3 and the rest are the hyperedges of size 2 of two initial ideal graphs which do not contain the vertex $45$. Since $46$ and $57$ are crossing edges in the complete graph on 4 vertices, they are not vertices of the glued hypergraph.

\begin{figure}[ht]
\begin{tikzpicture}[scale=2]
  \graph [nodes={circle, draw, inner sep=1pt, minimum size=1em}, clockwise, radius=1.7cm, n=10] {
    34/"34",
    35/"35",
    45/"45",
    12/"12",
    13/"13",
    14/"14",
    15/"15",
    23/"23",
    24/"24",
    25/"25"
  };
  \node[circle, draw, inner sep=1pt, minimum size=1em] (16) at (-2.2,1) {16};
  \node[circle, draw, inner sep=1pt, minimum size=1em] (17) at (-2.2,0.6) {17};
  \node[circle, draw, inner sep=1pt, minimum size=1em] (26) at (-2.2,0.2) {26};
  \node[circle, draw, inner sep=1pt, minimum size=1em] (27) at (-2.2,-0.2) {27};
  \node[circle, draw, inner sep=1pt, minimum size=1em] (36) at (-2.2,-0.6) {36};
  \node[circle, draw, inner sep=1pt, minimum size=1em] (37) at (-2.2,-1) {37};
  \draw (12) -- (34);
  \draw (12) -- (35);
   \path (12) edge [bend right, MyGreen, line width=3]  (45);
  \path (13) edge [bend left, MyBlue, line width=3]  (45);
  \path (23) edge [MyRed, line width=3]  (45);
  \draw (14) -- (23);
  \draw (12) -- (34);
  \draw (15) -- (23);
  \draw (15) -- (24);  
  \draw (15) -- (34);
  \draw (25) -- (34);
  
  \begin{scope}[shift={(1.545,-0.166)}]
    \graph [nodes={circle, draw, inner sep=1pt, minimum size=1em}, clockwise, radius=1.7cm, n=6] {
      67/"67",
      57/"\textcolor{gray}{57}"[dashed,draw=gray],
      56/"56",
      47/"47",
      46/"\textcolor{gray}{46}"[dashed,draw=gray],
      45/"45"
    };

    \draw (47) -- (56);
    \path (45) edge [MyRed,line width=3] (67);
    \path (45) edge [bend left, MyBlue,line width=3] (67);
    \draw (45) edge [bend right, color=MyGreen,line width=3] (67);
  \end{scope}
\end{tikzpicture}

\caption{Glued hypergraph identified via 7 observable nodes.}
\label{fig: identifying initial ideal graphs}
\end{figure}
\end{example}

Let $I_1$ and $I_2$ be two ideals in a polynomial ring $\mathbb{R}[\textbf{x}]:=\mathbb{R}[x_1, \dots, x_n]$. We now  recall the definition of join of $I_1$ and $I_2$ from \cite{combinatorialsecant}. Introduce $2n$ new unknowns, grouped in 2 vectors $\textbf{y}_j = (y_{j1}, \ldots, y_{jn})$, $j \in \{1,2\}$ and consider the polynomial ring $\mathbb{R}[\textbf{x},\textbf{y}]$ in $2n+n$ variables. Moreover, let $I_j(\textbf{y}_j)$ be the image of the ideal $I_j$ in $\mathbb{R}[\textbf{x},\textbf{y}]$ under the map $ \textbf{x}\mapsto \textbf{y}_j$. Then the \emph{join} $I_1 \ast I_2$ is the elimination ideal 
$$\left(I_1(\textbf{y}_1) + I_2(\textbf{y}_2) + \langle y_{1i} + y_{2i} - x_i \ | \  1\leq i \leq n \rangle \right) \cap \mathbb{R}[\textbf{x}].$$
Given a factor analysis graph $G=(V \cup \mathcal{H}, D)$ with $|\mathcal{H}|=2$ latent nodes, we can identify it with two one-factor analysis graphs. By definition, we have that the ideal of invariants of the two-factor analysis model is equal to the join of the ideals of the one-factor models, that is, $I(G) = I_{A_1,B_1,1} \ast I_{A_2, B_2,1}$. In this section, we find a Gr\"obner basis of this join ideal if $ |\! \ch(h_1) \cap \ch(h_2)| = |A_1 \cap A_2| = 2$. We assume that $p\geq 4$, since the ideal of invariants is otherwise empty.\\

By \cite[Theorem 2.3]{SIMIS20001}, for any term order $\prec$, and any two ideals $I_1,I_2$, we have that $\initial_{\prec} (I_1 \ast I_2) \subseteq \initial_{\prec}(I_1) \ast \initial_{\prec}(I_2)$. Thus, if we find a collection of polynomials $\mathcal{G} \subset I_1 \ast I_2$ such that $\langle \initial_{\prec}(g) \ | \ g \in \mathcal{G} \rangle = \initial_{\prec}(I_1) \ast \initial_{\prec}(I_2)$, then we can deduce that $\mathcal{G}$ is a Gröbner basis with respect to the term order $\prec$ for $I_1 \ast I_2$. A term order $\prec$ is called \emph{2-delightful} for two ideals $I_1$ and $I_2$, when the equality $\initial_{\prec}(I_1 \ast I_2) = \initial_{\prec}(I_1) \ast \initial_{\prec}(I_2)$ holds. We next describe the join of the initial ideals of two sparse one-factor analysis models with overlap two with respect to any circular term order. 

\begin{lemma}\label{lem: initial ideal of a hypergraph}
    Let $I_{A_1, B_1,1}$ and $I_{A_2, B_2, 1}$ be the toric ideals of invariants of two sparse one-factor models with $|A_1 \cap A_2|=2$ and $\prec$ be any circular term order. Then $\initial_{\prec}(I_{A_1, B_1,1})\ast \initial_{\prec}(I_{A_2, B_2, 1})$ is the monomial edge ideal of the glued hypergraph $G_{A_1, B_2} \times_{V} G_{A_2,B_2}$.
\end{lemma}

\begin{proof}
Let $I= \initial_{\prec}(I_{A_1, B_1,1})$ and $J= \initial_{\prec}(I_{A_2, B_2, 1})$. Consider the irreducible component decomposition of $I = \bigcap I_{\nu}$ and $J = \bigcap J_{\mu}$. Since $I$ and $J$ are monomial edge ideals, the irreducible components are the minimal vertex covers of $G_{A_1,B_1}$ and $G_{A_2,B_2}$, including the isolated vertices (\cite[Corollary 1.35]{van2013beginner}). We use the Alexander duality formula for the join of monomial ideals from \cite[Theorem 2.6]{combinatorialsecant}. Note that the vertex set $V(G_{V \backslash (A_1 \cap A_2),\emptyset})$ is the common isolated vertices of $G_{A_1,B_1}$ and $G_{A_2,B_2}$. By \cite[Lemma 2.3]{combinatorialsecant}, we obtain that the indices of the square-free irreducible monomial ideals $I_{\nu} \ast J_{\mu}$ is in the following form: a minimal vertex cover of $G_{A_1,B_1}$, a minimal vertex cover of $G_{A_2,B_2}$, $V(G_{V \backslash (A_1 \cap A_2),\emptyset})$. This corresponds to taking the intersection of generators of $I_{\nu}$ and $J_{\mu}$ as sets.

Note that this collection of vertices covers all minimal vertex covers of the glued hypergraph $G_{A_1, B_2} \times_{V} G_{A_2,B_2}$. By setting up the facets of the associated simplicial complex to be the maximal independent sets of the glued hypergraph, by \cite[Lemma 1.5.4]{herzog2011monomial}, we conclude that $\bigcap_{\nu,\mu} (I_{\nu} \ast J_{\mu})$ is the monomial edge ideal of the glued hypergraph.
\end{proof}

We next construct a Gröbner basis for two factor models with overlap two, that is, $|A_1 \cap A_2| =2$. We find a collection of polynomials $\mathcal{G}$ such that their initial terms with respect to any circular term order are the generators of the monomial edge ideal of the glued hypergraph, and thus this collection forms a Gröbner basis. We discuss the application of the ``2-delightful" strategy to cases where $|A_1 \cap A_2|\geq 3$ in Section~\ref{subsec:larger-intersections}.
We refer to the degree 3 generators below as \emph{hexads} because they are polynomials obtained from six observed random variables, analogous to tetrads, which are degree 2 generators obtained from four observed random variables.

\begin{theorem}\label{thm: a Gröbby}
The generators of a Gröbner basis for $I_{A_1,B_1,1} \ast I_{A_2, B_2,1}$ with respect to any circular term order for sparse two-factor  analysis models where $A_1 \cap A_2 = \{j_1,j_2\}$ comes in three types:
\begin{enumerate}
\item Degree-one monomial: $\sigma_{ik}$ is a generator, where $\pa(i) \cap \pa(j) = \emptyset$.
\item Tetrads: The binomial generators of the Gröbner basis of $I_{A_1,B_1,1}$ and $I_{A_2, B_2,1}$ with respect to any circular order that do not contain $\sigma_{j_1 j_2}$.
\item Hexads: Consider $i_1, i_2 \in A_1 \backslash \{j_1,j_2\}$ and $k_1,k_2 \in A_2 \backslash \{j_1,j_2\}$. Then 
$$\underline{\sigma_{k_1 k_2} \sigma_{i_1 i_2} \sigma_{j_1 j_2}} - \sigma_{k_1 k_2} \sigma_{j_1 i_2} \sigma_{j_2 i_1} -\sigma_{i_1 i_2} \sigma_{j_1 k_2} \sigma_{j_2 k_1},$$ is a degree three generator, where $\{i_1, i_2\}, \{j_1, j_2\}$ and $\{j_1 ,j_2\}, \{k_1, k_2\}$ are non-crossing edges of the complete graphs on the vertices $A_1 \backslash  \{j_1, j_2\}$ and $A_2 \backslash  \{j_1, j_2\}$ respectively.
\end{enumerate}
\end{theorem}
\begin{proof}
 This set of polynomials is in the join by the combinatorial definition of join of ideals. And by the previous lemma, their initials are exactly those \cite[Theorem 2.3]{SIMIS20001}. This concludes the proof. 
\end{proof} 
The theorem implies that if the set of children of one latent node is strictly contained in the other, then the generators consist of degree-one monomials and tetrads, and thus it is a toric ideal. This means equivalently that the (two) children of exactly one latent node are non-pure. 

\begin{corollary} \label{cor:proposed-equals-vanishing}
Let $|A_1 \cap A_2| =2$ for a sparse two-factor analysis model graph. Then the ideal $(M_{p,2} +  M_{\leq 1}(G)) \cap \mathbb{R}[\sigma_{ij},\ i < j]$ is equal to the join $I(G) = I_{A_1,B_1,1} \ast I_{A_2, B_2,1}$, and thus prime. In particular, $I(G)$ is generated by all off-diagonal minors of size at most 3 in $M_{p,2}$ and off-diagonal minors of size at most 2 in $M_{\leq 1}(G)$. 
\end{corollary}

\begin{proof}

Let $J = (M_{p,2} +  M_{\leq 1}(G)) \cap \mathbb{R}[\sigma_{ij},\ i < j]$. By Theorem~\ref{thm:variety-of-model}, $J \subseteq I_{A_1,B_1,1} \ast I_{A_2, B_2,1}$. Thus it is enough to show $I:= I_{A_1,B_1,1} \ast I_{A_2, B_2,1} \subseteq J$. The degree-one and degree-two generators of $I_{A_1,B_1,1} \ast I_{A_2, B_2,1}$ are in $M_{\leq 1}(G)$. The degree three generators can be described as a $3 \times 3$ off-diagonal minor.
Since $I_1$ and $I_2$ are both prime ideals, $I$ is also prime \cite[Proposition 1.2]{SIMIS20001}.
\end{proof}
\begin{example}
Consider the sparse two-factor analysis graph $G$ with $A_1 \cap A_2 = \{4,5\}$ from Figure~\ref{fig:graph_5_2_4}. The degree-one monomials of the generators of the Gröbner basis constructed in Theorem~\ref{thm: a Gröbby} are $\sigma_{16}, \sigma_{17}, \sigma_{26}, \sigma_{27},\sigma_{36}, \sigma_{37}$. These are the same as the set $M_{\leq 1}(G)$. The tetrads are all the ones in form (\ref{eq: tetrads}) that do not contain the $\sigma_{45}$. For example, $\underline{\color{MyBlue}\sigma_{12} \sigma_{45}} - {\color{MyRed}\sigma_{14} \sigma_{25}}$ from Example~\ref{ex: toric edge ideal of the complete graph with isolated vertices} is not a tetrad generator of the Gröbner basis for the ideal of the sparse two-factor analysis. Finally, we obtain the following three hexads to form the Gröbner basis: 
\begin{align*}
\begin{split}
&\sigma_{67}\sigma_{12}\sigma_{45} - \sigma_{67}\sigma_{24}\sigma_{15} - \sigma_{12}\sigma_{47}\sigma_{56},\\
&\sigma_{67}\sigma_{13}\sigma_{45} - \sigma_{67} \sigma_{34} \sigma_{15} - \sigma_{13}\sigma_{47} \sigma_{56}, \\
&\sigma_{67} \sigma_{23} \sigma_{45} - \sigma_{67} \sigma_{34} \sigma_{25} - \sigma_{23} \sigma_{47} \sigma_{56}. 
\end{split}
\end{align*}
\end{example}

\section{Conclusion and Open Questions} \label{subsec:computational-results}

In this paper, we derived novel results on the algebro-geometric aspects of sparse factor analysis models. We first proved upper and lower bounds for the dimension. While the upper bound holds for arbitrary models, the lower bound holds for models that satisfy a minimal level of sparsity, which we formalize in the ZUTA condition. In many cases, upper and lower bounds coincide and one obtains a formula for the dimension. In particular, our study reveals that  sparse factor analysis models, unlike full factor analysis models, may not have expected dimension. Then, we studied the ideal of invariants of sparse factor models with two latent nodes. We presented an ideal that cuts out the model and, moreover, we derived a Gröbner basis for models with at most overlap two, i.e., models where at most two observed nodes have more than one latent parent. On a technical level, we extended the \emph{delightful} strategy, which was previously applied to secants, to joins of ideals.

In what follows, we outline some possible future directions and open questions that arose from our paper. All the examples below can be reproduced by our code on \href{https://mathrepo.mis.mpg.de/sparse-factor-analysis/}{MathRepo}.

\begin{table}[ht]
\centering
\begin{tabular}{|c | c | c | c|} 
 \hline
 intersection size & degree & \# indeterminates & \# monomials  \\ [0.5ex] 
 \hline
 3 & 3 & 8 & 4 \\ 
 \scriptsize{(p=7)} & 5 & 9 & 6 \\
 \hline
 4 & 5 & 11 & 6 \\
 \scriptsize{(p=8)} & 5 & 11 & 8 \\
  & 5 & 12 & 10 \\
  & 5 & 12 & 12 \\
 \hline
\end{tabular}

\caption{Polynomials in Gröbner bases for sparse two-factor models with children sets of overlap $3$ or $4$. Each line reports the degrees, number of indeterminates, and number of monomials of one type of polynomial in the Gröbner basis.}
\label{table:0}
\end{table}

\subsection{Sparse factor models with larger overlaps}
\label{subsec:larger-intersections}

The circular term order is not always 2-delightful for the ideal of invariants of sparse two-factor analysis models for examples where $|\ch(h_1) \cap \ch(h_2)| = |A_1 \cap A_2| \geq 3$. Consider the sparse two-factor analysis graph from Figure \ref{fig:graph_5_2_4} with the additional edge $h_2 \rightarrow v_3$. The generators of a Gröbner basis with respect to any circular term order of the join has degrees one, two, three (degree-one monomials, tetrads, and non-hexads) and five, whereas the join of initial ideals is generated by at most degree three generators. In particular, one of the generators has the form
\begin{align*}
 &\sigma_{45}\sigma_{67}\sigma_{57}\sigma_{14}\sigma_{36}-\sigma_{45}\sigma_{67}\sigma_{15}\sigma_{36}\sigma_{47}+\sigma_{56}\sigma_{67}\sigma_{35}\sigma_{14}\sigma_{47}\\
-&\sigma_{56}\sigma_{57}\sigma_{14}\sigma_{36}\sigma_{47}-\sigma_{67}\sigma_{35}\sigma_{46}\sigma_{57}\sigma_{14}+\sigma_{46}\sigma_{57}\sigma_{15}\sigma_{36}\sigma_{47}.  
\end{align*}
This is a polynomial of degree $5$ in $9$ indeterminates involving $6$ monomial terms. The monomial terms coincide with monomial terms in a pentad \cite{kelley1935essential, AlgebraicFactorAnalysis} although the pentad has twelve terms. The missing monomials of the pentad are reduced by the elements in $M_{\leq 1}(G)$. 
We list  the types of homogeneous polynomials when we compute Gröbner bases for larger intersections among the children sets of the latent variables in Table \ref{table:0}. 
Our computations support the following conjecture for factor analysis models with two latent variables.

\begin{conjecture} \label{conj:generators}
    The ideal of the sparse two-factor analysis model corresponding to graph $G$ is generated by off-diagonal $3\times3$-minors, pentads, and the polynomials in $M_{\leq 1}(G)$.
\end{conjecture}

Although the ``delightful strategy'' with circular term orders is helpful for Gröbner basis, we observed that it fails for sparse two-factor analysis models with $|\ch(h_1) \cap \ch(h_2)| \geq 3$. However, since by \cite[Prop.~2.4]{combinatorialsecant}, the join of monomial ideals is monomial, one may consider constructing the join $\initial_{\prec}(I_{A_1, B_1,1})\ast \initial_{\prec}(I_{A_2, B_2, 1})$ as the monomial edge ideal of another hypergraph with respect to another term order. 

\begin{question*}
    Is there a 2-delightful term order for sparse two-factor analysis models? In other words, is there a term order $\prec$ such that $\initial_{\prec}(I_{A_1,B_1,1}\ast I_{A_2, B_2, 1}) = \initial_{\prec}(I_{A_1, B_1,1})\ast \initial_{\prec}(I_{A_2, B_2, 1})$?
\end{question*}

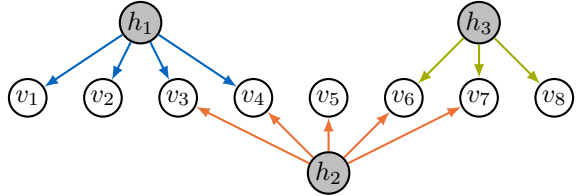
\begin{figure}[b]
\begin{tikzpicture}[align=center]
    \node[fill=lightgray] (h1) at (-2,1) {$h_1$};
    \node[fill=lightgray] (h2) at (0.5,-1) {$h_2$};
    \node[fill=lightgray] (h3) at (2.5,1) {$h_3$};
    
    \node[] (1) at (-3.5,0) {$v_1$};
    \node[] (2) at (-2.5,0) {$v_2$};
    \node[] (3) at (-1.5,0) {$v_3$};
    \node[] (4) at (-0.5,0) {$v_4$};
    \node[] (5) at (0.5,0) {$v_5$};
    \node[] (6) at (1.5,0) {$v_6$};
    \node[] (7) at (2.5,0) {$v_7$};
    \node[] (8) at (3.5,0) {$v_8$};    
    \draw[MyBlue] [-latex] (h1) edge (1);
    \draw[MyBlue] [-latex] (h1) edge (2);
    \draw[MyBlue] [-latex] (h1) edge (3);
    \draw[MyBlue] [-latex] (h1) edge (4);
    \draw[MyRed] [-latex] (h2) edge (3);
    \draw[MyRed] [-latex] (h2) edge (4);
    \draw[MyRed] [-latex] (h2) edge (5);
    \draw[MyRed] [-latex] (h2) edge (6);
    \draw[MyRed] [-latex] (h2) edge (7);
    \draw[MyGreen] [-latex] (h3) edge (6);
    \draw[MyGreen] [-latex] (h3) edge (7);
    \draw[MyGreen] [-latex] (h3) edge (8);
    \end{tikzpicture}
    \caption{A factor analysis graph with $3$ latent factors and two overlaps.}
\label{fig:latentare3}
\end{figure}
\subsection{Sparse factor models with more than two latent factors}
The generators introduced in Theorem \ref{thm: a Gröbby} can be partially used for sparse factor analysis models with more than two latent nodes where we have non-empty intersections only for consecutive intersections $|\ch(h_i) \cap \ch(h_{i+1})| = |A_i \cap A_{i+1}|=2$ with $i \in [m-1]$. Consider for instance the graph $G$ in Figure \ref{fig:latentare3}. Since the induced subgraph $G[\{h_3\} \cap V] \subset G$ gives rise to the toric ideal of $K_3$ and isolated vertices $[5]$, we obtain the generators of a Gröbner basis with respect to any circular term order in three types as follows: 
\begin{enumerate}
\item 11 degree-one monomials $\sigma_{ij}$, where $\pa(i) \cap \pa(j) = \emptyset$:
\[
\sigma_{15},\sigma_{16},\sigma_{17},\sigma_{18},\sigma_{25},\sigma_{26},\sigma_{27},\sigma_{28},\sigma_{38},\sigma_{48},\sigma_{58},
\]
\item 6 tetrads that do not contain $\sigma_{34}$ or $\sigma_{67}$:
\begin{align*}
\sigma_{47}\sigma_{56}&-\sigma_{46}\sigma_{57}, & \sigma_{37}\sigma_{56}&-\sigma_{36}\sigma_{57}, & \sigma_{37}\sigma_{46}&-\sigma_{36}\sigma_{47},\\ 
\sigma_{37}\sigma_{45}&-\sigma_{35}\sigma_{47}, & \sigma_{36}\sigma_{45}&-\sigma_{35}\sigma_{46},  & \sigma_{14}\sigma_{23}&-\sigma_{13}\sigma_{24},
\end{align*}
\item and 2 hexads:
\begin{equation*}
\sigma_{12}\sigma_{34}\sigma_{57}-\sigma_{12}\sigma_{35}\sigma_{47}-\sigma_{13}\sigma_{24}\sigma_{57}, \quad
\sigma_{12}\sigma_{34}\sigma_{56}-\sigma_{12}\sigma_{35}\sigma_{46}-\sigma_{13}\sigma_{24}\sigma_{56}.
\end{equation*}
\end{enumerate}
 If we add one more observable node $v_9$ and the edge $h_3 \rightarrow v_9$, we obtain degree-one monomials, tetrads, hexads, and a degree four generator in ten indeterminates which seems to have a combinatorial structure as in the hexad case, e.g.,
\begin{align*}
\sigma_{12}\sigma_{34}\sigma_{67}\sigma_{89}-\sigma_{12}\sigma_{34}\sigma_{68}\sigma_{79}-\sigma_{12}\sigma_{89}\sigma_{36}\sigma_{47}-\sigma_{67}\sigma_{89}\sigma_{13}\sigma_{24}+\sigma_{13}\sigma_{24}\sigma_{68}\sigma_{79} .  
\end{align*}
Moreover, we obtain that the circular term order is 3-delightful, i.e.,
\[
\initial_{\prec}(I_{A_1,B_1,1} \ast I_{A_2,B_2,1} \ast I_{A_3,B_3,1}) = \initial_{\prec}(I_{A_1,B_1,1}) \ast \initial_{\prec}(I_{A_2,B_2,1}) \ast \initial_{\prec}(I_{A_3,B_3,1}).
\]
Adding a fourth latent variable $h_4$ while keeping the cardinality of intersections $2$ will also give rise to polynomials of degree $5$ in $13$ indeterminates with $8$ summands, like
\begin{multline*}
    \sigma_{12}\sigma_{34}\sigma_{67}\sigma_{910}\sigma_{1112}-\sigma_{12}\sigma_{34}\sigma_{67}\sigma_{911}\sigma_{1012}-
    \sigma_{12}\sigma_{34}\sigma_{1112}\sigma_{69}\sigma_{710}-\sigma_{12}\sigma_{910}\sigma_{1112}\sigma_{36}\sigma_{47}+ \\
    \sigma_{12}\sigma_{911}\sigma_{1012}\sigma_{36}\sigma_{47}-\sigma_{67}\sigma_{910}\sigma_{1112}\sigma_{13}s_{24}+\sigma_{67}\sigma_{13}\sigma_{24}\sigma_{911}\sigma_{1012}+\sigma_{1112}\sigma_{13}\sigma_{24}\sigma_{69}\sigma_{710}.
\end{multline*}
We summarize the findings for the maximal degree polynomials we have computed based on the number of latent nodes in Table \ref{table:1}. 

\begin{table}[ht]
\centering
\begin{tabular}{|c | c | c | c|} 
 \hline
 \# latent nodes & degree & \# indeterminates & \# monomials  \\ [0.5ex] 
 \hline
 0 & 1 & 1 & 1 \\ 
 1 & 2 & 4 & 2 \\ 
 2 & 3 & 7 & 3 \\
 3 & 4 & 10 & 5 \\
 4 & 5 & 13 & 8 \\
 \hline
\end{tabular}

\caption{Degrees, number of variables and number of terms for different latent nodes when there are at most $2$ intersections.}
\label{table:1}
\end{table}

\begin{question*}
  Can we show that this behavior generalizes for more hidden variables with intersection two?
  That is, in the presence of $k$ latent nodes,
  is the polynomial of maximal degree a degree $k+1$ polynomial 
  in $3k+1$ variables
  which has $(k+2)$ Fibonacci number of terms? 
  In fact, using the delightful strategy, can we find a Gröbner basis with respect to any circular term order for sparse $k$-factor analysis models with more than two latent nodes where $|\ch(h_i) \cap \ch(h_{i+1})|=2$ for $i \in [m-1]$?
\end{question*}
Describing a Gröbner basis for this case would require a generalization of the glued hypergraph (Definition~\ref{def: new graph}) and of the construction of the polynomials as in Theorem~\ref{thm: a Gröbby}.

\section*{Acknowledgements}
This project was started at a workshop funded by the TUM Global Incentive Fund ``Algebraic Methods in Data Science''. It received funding from the European Research Council (ERC) under the European Union’s Horizon 2020 research and innovation programme (grant agreement No 883818) and from the German Federal Ministry of Education and Research and the Bavarian State
Ministry for Science and the Arts.  Nils Sturma acknowledges support from Munich Data Science Institute (MDSI) at the Technical University of Munich through the Linde/MDSI PhD Fellowship program.  The authors are grateful for helpful discussions with Danai Deligeorgaki, Alex Markham, Pratik Misra, Lisa Seccia, and Liam Solus.
\bibliographystyle{alpha}
\bibliography{bibliography}
\end{document}